\documentclass[11pt,reqno]{amsart}
\usepackage[utf8]{inputenc}
\usepackage{pgf,pgfarrows,pgfnodes,pgfautomata,pgfheaps,pgfshade,hyperref, amssymb,enumerate,amsmath}
\usepackage[all]{xy}
\usepackage[capitalize]{cleveref}
\usepackage{mathtools}
\usepackage[shortlabels]{enumitem}
\usepackage{stmaryrd}
\usepackage{fullpage}

\newtheorem{theorem}{Theorem}[section]
\newtheorem*{theorem*}{theorem}
\newtheorem{proposition}[theorem]{Proposition}
\newtheorem{lemma}[theorem]{Lemma}
\newtheorem{corollary}[theorem]{Corollary}

\newtheorem{rem}[theorem]{Remark}

\newtheorem{definition}[theorem]{Definition}

\usepackage{euscript,epsfig}
\usepackage{tikz}
\usetikzlibrary{graphs}
\usetikzlibrary[graphs]
\usetikzlibrary{arrows}
\usepackage{tikz-cd}
\usetikzlibrary{decorations.markings}
\usepackage[colorinlistoftodos]{todonotes}

\def\R{{\mathbb R}}
\def\Z{{\mathbb Z}}

\def\N{{\mathbb N}}

\def\QQ{{\mathbb Q}}

\def\cP{{\mathcal P}}

\def\cB{{\mathcal B}}

\def\cA{{\mathcal A}}
\def\cL{{\mathcal L}}

\def\cM{{\mathcal M}}

\def\cS{{\mathcal S}}

\def\cO{{\mathcal O}}

\newcommand{\bt}{\boldsymbol\tau}

\DeclareMathOperator{\Id}{Id}
\DeclareMathOperator{\diag}{diag}
\DeclareMathOperator{\Aff}{Aff}

\DeclareMathOperator{\Inf}{Inf}
\DeclareMathOperator{\sub}{sub}

\author[P.~Cecchi--Bernales]{Paulina Cecchi--Bernales}
\address{Departamento de Matem\'aticas, Universidad de Chile. Las Palmeras 3425, \~Nu\~noa, Chile.}
\email{pcecchi@uchile.cl}
\thanks{P. Cecchi--Bernales was supported by ANID/Fondecyt Iniciación
11240886}

\author[S.~Donoso]{Sebasti\'an Donoso}
\address{Departamento de Ingenier\'{\i}a Matem\'atica and Centro de Modelamiento Matem\'atico, Universidad de Chile \& UMI-CNRS 2807. Beauchef 851, Santiago, Chile.}
\email{sdonosof@uchile.cl}
\thanks{S.~Donoso was supported by ANID/Fondecyt/1241346 and  by ANID PIA/BASAL FB210005 (Centro de Modelamiento Matemático).
}

\begin{document}

\title{Entropy and complexity in a strong orbit equivalence class via pseudo-Toeplitz subshifts} 



\begin{abstract}
We develop a method for constructing minimal subshifts with a prescribed bound on the factor complexity within any strong orbit equivalence (SOE) class. This implies realizations of topological entropies within any SOE class, strengthening a series of classical results by Sugisaki. While prior work established complexity controls for zero-entropy systems, the approach of the present work extends to positive-entropy regimes. Our construction introduces the class of pseudo-Toeplitz subshifts, and we show that they exist in any SOE class. More precisely, for any $\alpha \geq 1$ and sequence $g_n$ growing exponentially at rate $\log(\alpha)$, we construct a system in this class whose topological entropy is $\log(\alpha)$ and whose complexity grows strictly faster, or, under certain conditions, slower than $g_n$. As a consequence, any Choquet simplex can be realized as the set of invariant measures of a Toeplitz subshift with precisely controlled complexity, in either a zero or positive-entropy regime. 
\end{abstract}

\date{\today}

\maketitle

 \section{Introduction}

A Cantor system is a pair $(X,T)$ where $X$ is a Cantor space, and $T$ is a homeomorphism on $X$. Cantor systems form an important class of dynamical systems and include a vast range of examples, notably subshifts. For the class of minimal Cantor systems, there is a rich interplay between the algebraic properties of associated objects, such as their dimension group, and purely dynamical properties such as strong orbit equivalence. In their very influential paper, Giordano, Putnam, and Skau proved in \cite [Theorem 2.1]{Giordano_Putman_Skau_Topological_orbit_equiv_crossed_products:1995} that two Cantor minimal systems are strongly orbit equivalent if and only if they have isomorphic dimension groups with unit (we refer to \cref{subsec:soeanddg} for the precise definitions).

In contrast to what was believed when the notion of strong orbit equivalence was introduced, works have shown that strong orbit equivalence imposes no restriction on entropy and vice versa. For instance, Boyle and Handelman \cite{Boyle_Handelman_entropy_vs_oe:1994} showed that within any SOE class there exists a zero-entropy system. This result was extended by Sugisaki in a series of papers \cite{Sugisaki_entropy_SOE_homeo_I:2003, Sugisaki_subshift_within_SOE:2007, Sugisaki_entropy_SOE_homeo_II:1998, Sugisaki_toeplitz_bratteli_SOE:2001} to show that in any SOE class, any entropy can be realized. Regarding the complexity function $p_X(n)$, whose exponential growth rate accounts for the topological entropy, there are some natural restrictions on the SOE class. An important one is having a non-superlinear complexity function, meaning that $\liminf_{n\to\infty}\frac{p_X(n)}{n}<\infty$. In this case, the dimension group is an abelian group whose rational rank is finite. It was proved in \cite{Cecchi_Donoso_SOE_superlinear_complexity:2022} that non-superlinearity is, in fact, the only hard restriction on complexity within an SOE class. Specifically, we established that for any simple dimension group with unit $(G,G^+,u)$ and any sequence of positive numbers $(a_n)_{n\in\N}$ such that $\lim_{n\to\infty} n/a_n=0$, there exists a minimal subshift whose dimension group is order isomorphic to $(G,G^+,u)$ and whose complexity function $p_X(n)$ grows slower than $a_n$, meaning $\lim_{n\to\infty} p_X(n)/a_n=0$. As a consequence, it was shown that for any superlinear function one can construct, within any SOE class, a minimal subshift whose complexity function is dominated by the given function. In particular, any Choquet simplex of invariant measures can be realized with a minimal subshift of arbitrarily low superlinear complexity, extending a result by Cyr and Kra \cite{Cyr_Kra_erg_prop_zero_entropy:2020}, where they showed that the number of ergodic measures could be uncountable for arbitrarily low superlinear complexity functions. This also extends a result by Gjerde and Johansen \cite{Gjerde_Johansen_Bratteli_Toeplitz:2000}, where they proved that for any Choquet simplex $K$ there exists a zero-entropy Toeplitz subshift whose set of invariant measures is affine homeomorphic to $K$. Note that these results are mostly about zero-entropy systems. 

In this paper, we study the relationship between complexity and SOE by considering different growth rates of the complexity function. We develop a method to realize these systems within any SOE class, controlling the complexity function even in high-complexity regimes. This strengthens Sugisaki's results and, in the authors' opinion, provides a simpler, more conceptual proof of them. We also believe our approach could lead to further applications. In particular, we believe that the technique we describe here to construct subshifts with arbitrarily high entropy, could be slightly adapted and combined with some results in \cite{Sugisaki_almost11_embeddings_dimension_groups:2011} to get, inside any SOE class, a sequence of minimal subshifts with strictly increasing entropy, which would lead to a different proof of the fact that within any SOE class there exists a minimal Cantor system with infinite entropy \cite{Sugisaki_entropy_SOE_homeo_II:1998}.

We introduce the notion of {\em pseudo-Toeplitz} subshifts (see \Cref{def:pseudotoeplitz}), which correspond to minimal subshifts admitting a Bratteli-Vershik representation in whose associated Bratteli diagram the ratio between the maximal and minimal number of paths from the top vertex to a given level tends to $1$. Bratteli diagrams satisfying this property are called {\em pseudo-ERS} (see \Cref{def:ersandpseudoers}). This is inspired by the characterization of Toeplitz subshifts given in \cite{Gjerde_Johansen_Bratteli_Toeplitz:2000}, where they are identified up to conjugacy with the family of expansive Bratteli-Vershik systems associated with the {\em equal path number property} (also called {\em ERS\footnote{This stands for equal row sum.} property}), meaning that on each level of the associated Bratteli diagram, the number of paths between the top vertex and any vertex of the level is constant. In \cite{Sugisaki_toeplitz_bratteli_SOE:2001}, Sugisaki proved that in the SOE class of any Bratteli-Vershik system associated with the equal path number property there exists a Toeplitz subshift. We prove in \Cref{theo:pseudoers} that for any simple Bratteli diagram $(V,E)$, there exists a simple Bratteli diagram $(V',E')$ which verifies $K_0(V,E)\cong K_0(V',E')$ and which is pseudo-ERS. Furthermore, we prove that, within certain margins, the rate at which the ratio between the maximal and minimal number of paths from the top vertex to each level approaches $1$ can be bounded from below (see condition (iii) in the conclusion of that theorem). We then use \Cref{theo:pseudoers} to construct in \Cref{theo:low} and \Cref{theo:main} pseudo-Toeplitz subshifts inside any SOE class, with prescribed values of topological entropy and different bounds on the factor complexity function. In both results, as a consequence of the construction, if the starting SOE class contains a Bratteli Vershik system with the equal path number property, the resulting system is a Toeplitz subshift. This implies, in particular, that any Choquet simplex can be realized as the set of invariant measures of a Toeplitz subshift with precisely controlled complexity, in either a zero or positive-entropy regime.\\

\subsection{Statements of main results}

Throughout the article, $\N$ denotes the set of non-negative integers and $\N^\star$ denotes the set of positive integers. All logarithms are in base $e$. Recall that, for any two real valued functions $f,g$ defined on $\N$ or $\N^\star$, the notation $f(n)=O(g(n))$ means that $|f|$ is bounded above by $g$ up to a constant, that is, there exists $k>0$ such that for all $n$ large, $|f(n)|\leq kg(n)$. Similarly, $f(n)=\Omega(f(n))$ means that $|f|$ is not dominated by $g$ asymptotically, that is, there exists $k>0$ such that for all $n$ large, $|f(n)|\geq kg(n)$. In both cases, when the constant $k$ depends on some parameters $t_1,t_2, \ldots , t_m$, we use the notations $f(n)=O_{t_1,\ldots, t_m}(g(n))$ and $f(n)=\Omega_{t_1,\ldots, t_m}(g(n))$.\\

The main results of the article are the following.

\begin{theorem}\label{theo:pseudoers}
Let $(V,E)$ be a simple Bratteli diagram, let $h\colon \N^\star \to \R_+$ be a decreasing function satisfying $\lim_{i\to\infty}h(i)=0$ and $\lim_{i\to\infty}ih(i)=+\infty$. Let $f\colon \N^\star\to\R_+$ be any increasing function. Then, there exists a simple Bratteli diagram $(V',E')$ with sequence of incidence matrices $(M'_i)_{i\in\N}$ verifying the following conditions:
\begin{itemize}
    \item [(i)] $K_0(V,E)\cong K_0(V',E')$,
    \item [(ii)]for all $i\geq 1$, $u\in V'_{i-1}$ and $v\in V'_i$,
    $$M'_{i-1}(v,u)\geq f(|V'_i|), \text{ and }$$
    \item [(iii)]for all $i\geq 1$, $\frac{\ell'_i}{L'_i}\geq 1-h(L_i')$,
\end{itemize}
where $L_i'$ and $\ell_i'$ denote the maximal and minimal number of paths from the top vertex to the vertices in $V_i'$, respectively. In particular, $(V',E')$ is pseudo-ERS.
\end{theorem}

The following result shows that for any $\alpha > 1$ and $g_n$ a sequence verifying $\lim_{n\to\infty}\frac{\log(g_n)}{n}=\log(\alpha)$, under certain conditions on the speed of convergence of $\frac{\log(g_n)}{n}$ towards $\log(\alpha)$, any SOE class contains a pseudo-Toeplitz subshift whose factor complexity grows slower than $g_n$ and whose topological entropy equals $\log(\alpha)$.

\begin{theorem}\label{theo:low}
Let $(X,T)$ be a minimal Cantor system. Let $\alpha > 1$ and let $(g_n)_{n\in\N}$ be a sequence of positive real numbers such that $\lim_{n\to\infty}\frac{\log(g_n)}{n}=\log(\alpha)$. For each $n\geq 1$, let $C_n=\frac{\log(g_n)}{n}-\log(\alpha)$. Suppose that $\log(g_n)/n$ has a strictly decreasing subsequence and there exists a function $F\colon \N^\star \to\N^\star$ with $\lim_{i\to\infty}F(i)\nearrow +\infty$ and $\lim_{i\to\infty}\frac{F(i)}{i}=0$, such that there exists $K>2\log(\alpha)+2$ which verifies
\begin{equation}\label{eq:suffcondlow}
    F(n)C_{nF(n)}\geq K \quad \text{for all large enough } n.
\end{equation}
Then, there exists a pseudo-Toeplitz shift $(Y,S)$ such that 
\begin{itemize}
    \item $(X,T)$ and $(Y,S)$ are strong orbit equivalent,
    \item $\liminf_{n\to\infty}\frac{p_Y(n)}{g_n}=0$, and 
    \item $h(Y,S)=\log(\alpha)$.
\end{itemize}

\end{theorem}

\begin{rem}\label{rem:decreasing}
It is worth noting that the existence of a decreasing subsequence of $\log(g_n)/n$ is a necessary condition for the conclusion of \Cref{theo:low}. Indeed, second and third properties of $(Y,S)$ in the statement of the theorem imply that there exists a sequence $b_n$ such that $\log(p_Y(b_n))< \log(g_{b_n})$ and as the entropy $\log(\alpha)$ is the infimum of $\log(p_Y(n))/n$, we necessarily have that $\log(g_{b_n})/b_n>\log(\alpha)$. So, for a subsequence, we have $\log(g_n)/n\searrow \log(\alpha)$. The same argument proves that $b_nC_{b_n}$ goes to $+\infty$ as $n$ goes to $\infty$. Condition \eqref{eq:suffcondlow}, in contrast, is strictly stronger than this necessary condition on the sequence $(C_n)_{n\in\N}$ and it arises from our technique used to prove the theorem. We do not know if it is a necessary condition. Note that this technical condition is satisfied if $C_N=\Omega(1/N^\beta)$ for some $\beta<1/2$, and implies that $C_N\geq K/\sqrt{N}$ along $N=nF(n)$, which forces $g_n$ to grow faster than $e^{n(\log(\alpha)+1/\sqrt{n})}$. We remark that the zero-entropy version of \Cref{theo:low}, which corresponds to $\alpha=1$ and $g_n$ subexponential, is already treated in \cite[Theorem 1.2]{Cecchi_Donoso_SOE_superlinear_complexity:2022}.

\end{rem}

In the opposite direction, the following result establishes that within any SOE class, for any $\alpha\geq 1$ and for any sequence $g_n$ verifying $\lim_{n\to\infty}\frac{\log(g_n)}{n}=\log(\alpha)$, there exists a pseudo-Toeplitz subshift whose factor complexity grows faster than $g_n$ and whose topological entropy equals $\log(\alpha)$.

\begin{theorem}\label{theo:main}
Let $(X,T)$ be a minimal Cantor system. Let $\alpha \geq 1$ and let $(g_n)_{n\in\N}$ be a sequence of positive real numbers such that $\lim_{n\to\infty}\frac{\log(g_n)}{n}=\log(\alpha)$. Then, there exists a pseudo-Toeplitz shift $(Y,S)$ such that 
\begin{itemize}
    \item $(X,T)$ and $(Y,S)$ are strong orbit equivalent,
    \item $\liminf_{n\to\infty}\frac{g_n}{p_Y(n)}=0$, and 
    \item $h(Y,S)=\log(\alpha)$.
\end{itemize}

\end{theorem}
 
 The following corollary is a consequence of \Cref{theo:low} and \Cref{lem:divisibleers}.

\begin{corollary}\label{cor:alphaentropymeasureslow}
Let $K$ be a Choquet simplex and $\alpha> 1$ be given. Let $(g_n)_{n\in \N}$ be a sequence of positive real numbers satisfying the hypotheses of \Cref{theo:low}. Then, there exists a Toeplitz subshift $(Y,S)$ such that
\begin{enumerate}
    \item $\cM(Y,S)$ is affine homeomorphic to $K$. 
    \item $h(Y,S)=\log(\alpha)$.
    \item The complexity of $(Y,S)$ satisfies $\liminf \frac{p_Y(n)}{g_n}=0$. 
\end{enumerate}
\end{corollary}

 The following corollary is a consequence of \Cref{theo:main}.

\begin{corollary}\label{cor:zeroentropysoe}
Let $(X,T)$ be a Cantor minimal system. Let $g_n$ be a sequence of positive real numbers which is subexponential. Then there exists a zero-entropy pseudo-Toeplitz subshift $(Y,S)$ such that 
\begin{enumerate}
    \item $(Y,S)$ and $(X,T)$ are strong orbit equivalent.
    \item The complexity of $(Y,S)$ satisfies $\liminf \frac{g_n}{p_{Y}(n)}=0$. 
\end{enumerate}
\end{corollary}

The following corollaries are consequences of \Cref{theo:main} and \Cref{lem:divisibleers}.

\begin{corollary}\label{cor:alphaentropymeasures}
Let $K$ be a Choquet simplex and $\alpha\geq 1$ be given. Let $(g_n)_{n\in \N}$ be a sequence of positive real numbers satisfying $\lim \frac{\log(g_n)}{n}=\log(\alpha)$. Then, there exists a Toeplitz subshift $(Y,S)$ such that
\begin{enumerate}
    \item $\cM(Y,S)$ is affine homeomorphic to $K$. 
    \item $h(Y,S)=\log(\alpha)$.
    \item The complexity of $(Y,S)$ satisfies $\liminf \frac{g_n}{p_{Y}(n)}=0$. 
\end{enumerate}
\end{corollary}

\begin{corollary}\label{cor:zeroentropymeasures}
Let $K$ be a Choquet simplex. Let $(g_n)_{n\in \N}$ be a sequence of positive real numbers which is subexponential. Then, there exists a zero-entropy Toeplitz subshift $(Y,S)$ whose complexity $p_Y$ satisfies $\liminf \frac{g_n}{p_{Y}(n)}=0$ and whose set of invariant measures $\cM(Y,S)$ is affine homeomorphic to $K$. 
\end{corollary}

\subsection{Organization.} \Cref{sec:preliminaries} is devoted to recalling basic notions on minimal Cantor systems, symbolic dynamics, strong orbit equivalence, dimension groups, factor complexity and entropy. In \Cref{sec:BDandSadic} we introduce the main theoretical tools we use to prove our results, namely Bratelli diagrams, Bratteli-Vershik systems and their associated $S$-adic subshifts. In \Cref{sec:splitting} we describe in detail the technique we use to modify Bratteli diagrams in order to create entropy while keeping the dimension group, which we call the {\em splitting procedure}. Finally, \Cref{sec:proofs} is devoted to proving the main results and their corollaries. 

\subsection*{AI disclosure}
All the mathematical content of this paper was done by the authors. Prior to submitting the paper, an AI tool was used to correct typos and identify potential errors. 

\section{preliminaries}\label{sec:preliminaries}

\subsection{Minimal Cantor systems}\label{subsec:minimalcantor}
Throughout this text, a {\em (topological) dynamical system} is a pair $(X,T)$, where $X$ is a non-empty compact metric space and $T$ is a homeomorphism acting on $X$. 
We say that $(X,T)$ is a {\em Cantor system} if $X$ is a Cantor space (non-empty, compact, totally disconnected and without isolated points). The system is said to be {\em minimal} if $X$ contains no non-trivial $T$-invariant closed subsets, or, equivalently, if for each $x\in X$, the orbit $\{T^n(x): n\in \Z\}$ is dense in $X$.

Given a topological dynamical system $(X,T)$, an {\em invariant measure} of $(X,T)$ is a Borel probability measure on $X$ such that for every Borel subset $A\subseteq X$, $\mu(T^{-1}(A))=\mu(A)$. The set of all invariant measures of $(X,T)$ is denoted $\cM(X,T)$. We say that a measure $\mu\in\cM(X,T)$ is {\em ergodic} if whenever $T(A)=A$ for some Borel set $A\subseteq X$, either $\mu(A)=0$ or $\mu(A)=1$. The system $(X,T)$ is called {\em uniquely ergodic} if $\cM(X,T)$ consists of a single element. The set $\cM(X,T)$ is known to be non-empty and a {\em Choquet simplex} \cite{Bogoliouboff_Kryloff_1937}, that is, a compact convex metrizable subset $K$ of a locally convex real vector space such that for each $v\in K$ there is a unique probability measure $m$ supported on the extreme points of $K$ with $\int_{ext(K)}xdm(x)=v$. The extreme points of $\cM(X,T)$ correspond to the ergodic measures on $X$ (see for instance \cite{Glasner_ergodic_theory_joinings:2003}). Here, $\cM(X,T)$ is considered as a subspace of the topological dual of $C(X,\R)$ endowed with the weak$^\star$ topology.

\subsection{Subshifts, factor complexity and topological entropy}\label{subsec:subshifts}

A very important class of Cantor systems is given by subshifts. Let $\cA$ be a finite set of symbols with cardinality $|\cA|$ (called an {\em alphabet}). The space $\cA^{\Z}$ of sequences with symbols in $\cA$ endowed with the product topology is a Cantor space provided $|\cA|\geq 2$. Consider the {\em (left) shift transformation} $S$ on $\cA^{\Z}$ given by $S((x_n)_{n\in\Z}))=(x_{n+1})_{n\in\Z}.$
If $X\subseteq \cA^{\Z}$ is any closed and shift-invariant subset, the dynamical system $(X,S|_X)$, endowed with the induced topology, is denoted $(X,S)$ and is called a {\em subshift on $\cA$}. Finite concatenations of symbols in $\cA$ are called {\em words}. The empty word is denoted $\varepsilon$. Given a subshift $(X,S)$, a {\em factor} of $x\in X$ is a finite word appearing in $x$, that is, a word $w=x_kx_{k+1}\cdots x_{n}$ for some $n\geq k$ integers. The {\em language} of a sequence $x\in X$ is defined as the set of all factors of $x$ and is denoted $\cL_x$. The {\em language} $\cL_X$ of the subshift is the union $\cup_{x\in X}\cL_x$.\\   A sequence $x\in X$ is said to be {\em uniformly recurrent} if every factor of $x$ appears in an infinite number of positions separated by bounded gaps, and one has that a subshift is minimal if and only if every element of $X$ is uniformly recurrent, in which case $\cL_x=\cL_y$ for all $x,y\in X$ (see for instance \cite{Queffelec1987}).\\
Given an alphabet $\cA$, a sequence $x\in \cA^\Z$ is called a {\em Toeplitz sequence} if for all $k\in \N$ there exists $p\in \N$ such that $x_k=x_j$ whenever $k=j  \mod p$. A subshift $(X,S)$ is called a {\em Toeplitz subshift} if $X$ is the orbit closure under the shift transformation of some Toeplitz sequence. Toeplitz sequences are uniformly recurrent, which implies that Toeplitz subshifts are minimal (see for instance \cite{Downarowicz2005}). The following result shows that any Choquet simplex can be seen as the set of invariant measures of a Toeplitz subshift.
\begin{theorem}\cite[Theorem 5]{Downarowicz_Choquet_invariant_measures:1991}\label{theo:dow}
    Let $K$ be a Choquet simplex. There exists a Toeplitz subshift $(X,S)$ such that $\cM(X,S)$ is affine homeomorphic to $K$.
\end{theorem}

The {\it factor complexity function} or simply the {\it complexity} of a subshift $(X,S)$ is the function $p_X\colon \N\to\N$ given by the number of factors of length $n$ in $\cL_X$. From the subadditivity of $\log(p_X(n))$ and Fekete's lemma, the sequence $(\log(p_X(n))/n)$ converges to a real number which is denoted $h(X,S)$ and called the {\em topological entropy} of $(X,S)$. The number $h(X,S)$ in fact corresponds to the topological entropy defined for any topological dynamical system, see \cite[Chapter 7]{Walters82} for details.\\
The asymptotic behavior of the complexity function is known to restrict the Choquet simplex of invariant measures of a minimal subshift. More precisely, the following result tells us that for non-superlinear complexities, minimal subshifts can only support a finite number of ergodic invariant measures.
\begin{theorem}\cite[Corollary 1.3]{Boshernitzan:1984}
    Let $(X,S)$ be a minimal subshift and let $p_X$ denote the factor complexity function of $X$. If
$$\liminf_{n\to\infty}\frac{p_X(n)}{n}<\alpha<+\infty,$$
then the number of ergodic invariant probability measures of $(X,S)$ is at most $\max(\lfloor\alpha\rfloor,1)$, where $\lfloor\alpha\rfloor$ denotes the greatest integer which is $\leq \alpha$.
\end{theorem}

On the other hand, as the following result illustrates, topological entropy does not restrict the simplex of measures. This corresponds to the first part of \cite[Theorem 12]{Gjerde_Johansen_Bratteli_Toeplitz:2000}.

\begin{theorem}\cite[Theorem 12]{Gjerde_Johansen_Bratteli_Toeplitz:2000}\label{theo:noruegos}
    Let $K$ be any Choquet simplex. There exists a zero-entropy Toeplitz subshift $(X,S)$ such that K is affinely homeomorphic to the set $\cM (X, S)$ of invariant probability measures.
\end{theorem}

\Cref{theo:noruegos} is a generalization of \Cref{theo:dow}. In \cite[Theorem 1.4]{Cecchi_Donoso_SOE_superlinear_complexity:2022} we show that this generalization can be refined, by showing that for any Choquet simplex $K$ there exists a Toeplitz subshift whose set of invariant measures is affine homeomorphic to $K$ and whose complexity is arbitrarily close to linear. More precisely, we prove the following result.

\begin{theorem}\cite[Theorem 1.4]{Cecchi_Donoso_SOE_superlinear_complexity:2022}\label{theo:simplexprevious}
    Let $K$ be a Choquet simplex. Let $(p_n)_{n\in \N}$ be a sequence of positive real numbers such that $ \lim n/p_n=0$. Then, there exists a Toeplitz subshift $(X,S)$ whose complexity $p_X$ satisfies $\lim p_{X}(n)/p_n=0$
and whose set of invariant measures $\cM(X,S)$ is affine homeomorphic to $K$. 
\end{theorem}

\Cref{theo:simplexprevious} is also a generalization of the following result by Cyr and Kra.

\begin{theorem}\cite[Theorem 1.1]{Cyr_Kra_erg_prop_zero_entropy:2020}\label{theo:cyrkra}
    If $(p_n)_{n\in\N}$ is a sequence of natural numbers such that 
$$\liminf_{n\to\infty}\frac{p_n}{n}=\infty,$$ 
then there exists a minimal subshift $(X,S)$ which supports uncountably many ergodic measures and is such that 
$$\liminf_{n\to\infty}\frac{p_X(n)}{p_n}=0.$$

\end{theorem}

\begin{rem}\label{rem:generalizations}
    Note that both \Cref{cor:alphaentropymeasureslow} and \Cref{cor:alphaentropymeasures} are generalizations of \Cref{theo:simplexprevious} to the positive entropy regime.
\end{rem}

\subsection{Strong orbit equivalence and Dimension groups}\label{subsec:soeanddg}

Two dynamical systems $(X_1,T_1)$ and $(X_2,T_2)$ are {\it orbit equivalent} (OE for short) if there exists a homeomorphism $\phi:X_1 \to X_2$ such that for all $x \in X_1$,
$$\phi ( \{ T_1^n (x)  : n \in \Z \})= \{ T_2^n \phi (x) : n \in \Z\}.$$
Orbit equivalence between minimal systems determines the existence of two maps $n_1, n_2\colon X_1  \to \Z$ such that for all $x \in X_1$,
$$\phi \circ T_1 (x)= T_2^{n_1(x)} \circ \phi(x) \mbox { and }  \phi \circ T_1^{n_2(x)} (x)= T_2 \circ \phi (x).$$    
We say that $(X_1,T_1)$ and $(X_2,T_2)$ are {\it strong orbit equivalent} (SOE for short) if $n_1$ and $n_2$ both have at most one point  of discontinuity. 
It was shown in \cite{Boyle_top_orb_thesis:1983} that if $n_1$ (or $n_2$) is continuous, then the two systems are {\it flip conjugate}, that is, $(X_1,T_1)$ is either conjugate to $(X_2,T_2)$ or to $(X_2,T_2^{-1})$. When the spaces $X_1$ and $X_2$ are connected, OE and SOE are both the same as flip conjugacy. One of the most remarkable results on this subject is the complete characterization of SOE (resp. OE) for minimal Cantor systems by means of an algebraic invariant called the {\em dimension group} (resp. {\em reduced dimension group}) \cite{Giordano_Putman_Skau_Topological_orbit_equiv_crossed_products:1995}. Given a minimal Cantor system $(X,T)$, define the {\it coboundary map} $\beta\colon C(X,\Z)\to C(X,\Z)$ by $\beta (f)= f\circ T - f$, where $C(X,\Z)$ denotes the additive group of continuous functions from $X$ to $\Z$. The 
{\it dimension group} of $(X,T)$ is the following ordered group with unit
$$K^0(X,T)=(H(X,T), H^+(X,T), [1]),$$
where $H(X,T)=C(X,\Z)/\beta (C(X,\Z))$, $[\cdot]$ denotes the class modulo $\beta( C(X,\Z))$ of an element in $H(X,T)$, the positive cone $H^+(X,T)$ is the set of classes of non-negative functions and $1$ is the constant function equal to $1\in \Z$. It was proved in \cite{Herman1992} that the triple $K^0(X,T,\Z)$ is a {\em simple dimension group with unit} (see \cite{Durand_Perrin_Dimension_groups_dynamical_systems:2022} for more on abstract dimension groups). A {\em trace} of $K^0(X,T)$ is a group homomorphism $p:H(X,T)\to \R$ such that $p(H^+(X,T))\geq 0$ and $p([1])=1$. The {\em infinitesimal subgroup} of $K^0(X,T)$, denoted $\Inf(K^0(X,T))$, is the set of all classes $[f]\in H(X,T)$ such that $p([f])=0$ for every trace $p$. The quotient $H(X,T)/\Inf(K^0(X,T))$ with the induced positive cone also determines an ordered group with unit, which is called the {\em reduced dimension group} of the system, and is denoted $K^0(X,T)/\Inf(K^0(X,T))$. The following result is a characterization of OE and SOE in terms of the dimension group and reduced dimension group of a minimal Cantor system.

\begin{theorem}\cite[Theorem 2.1]{Giordano_Putman_Skau_Topological_orbit_equiv_crossed_products:1995}\label{theo:caracterizationsoe}
Let $(X_1,T_1)$ and $(X_2,T_2)$ be two minimal Cantor systems. Then, the following assertions hold,
\begin{itemize}
    \item $(X_1,T_1)$ and $(X_2,T_2)$ are SOE if and only if $K^0(X_1,T_1)$ and $K^0(X_2,T_2)$ are isomorphic.
    \item $(X_1,T_1)$ and $(X_2,T_2)$ are OE if and only if $K^0(X_1,T_1)/\Inf(K^0(X_1,T_1))$ and\\ $K^0(X_2,T_2)/\Inf(K^0(X_2,T_2))$ are isomorphic.
\end{itemize}
\end{theorem}
The set $\cM(X,T)$ corresponds to the set of traces of both $K^0(X,T)$ and $K^0(X,T)/\Inf(K^0(X,T))$ \cite{Herman1992}, and thus it is invariant under OE. Moreover, the set of traces is known to completely determine the positive cone: $H^+(X,T)=\{[f]\in H(X,T)\mid p([f])>0 \text{ for all trace } p\}\cup \{[0]\}$ (see \cite{Effros_dimension_C*_algebras:1981}).

\section{Bratteli diagrams and $\cS$-adic subshifts}\label{sec:BDandSadic}

A {\it Bratteli diagram} is an infinite directed graph $(V,E)$ with set of vertices $V=(V_i)_{i\geq 0}$ and set of edges $E=(E_i)_{i\geq 1}$ that can be written as a countable disjoint union of non-empty finite sets,
$$V=V_0\cup V_1\cup V_2\cup\cdots \quad \mbox{ and } \quad E=E_1\cup E_2\cup\cdots,$$
and with the property that $V_0$ is a single point and there exist a {\em range map} $r\colon E\to V$ and a {\em source map} $s\colon E\to V$ such that $r(E_i)\subseteq V_i$, $s(E_i)\subseteq V_{i-1}$ for all $i\geq 1$. Also, we assume that $s^{-1}(v)\neq \emptyset$ for all $v\in V$ and $r^{-1}(v)\neq \emptyset$ for all $v\in V\setminus V_0$.

For $i\in\N$, we let $M_i$ denote the $i$th {\em incidence matrix} of $(V,E)$, which is by definition the $|V_{i+1}|\times|V_i|$ matrix whose coefficient $M_i(k,j)$ is the number of edges connecting $u_j\in V_i$ with $v_k\in V_{i+1}$.

Given a strictly increasing sequence of natural numbers $(t_i)_{i\in\N}$, a new Bratteli diagram $(V',E')$ is defined by letting $V_i'=V_{t_i}$ and $M_i'=M_{t_{i+1}-1}M_{t_{i+1}-2}\cdots M_{t_i}$. The sets of edges $E_i'$ and the range and source maps are obtained from the new incidence matrices. The sequence $(t_i)_{i\in\N}$ is called a {\em sequence of telescoping depths} and the new diagram $(V',E')$ is called the {\em telescoping} of $(V,E)$ to $(t_i)_{i\in\N}$. If a Bratteli diagram $(V',E')$ can be telescoped to obtain $(V,E)$, then we say that $(V',E')$ is a {\em microscoping} of $(V,E)$. A Bratteli diagram is {\em simple} if it can be telescoped to obtain a new diagram all whose incidence matrices have positive entries. 

\begin{rem}\label{rem:preservesedgesvsvertices}
    Note that if a simple Bratteli diagram $(V,E)$ has the property that for all $i\geq 1$, for all $u\in V_i$ and $v\in V_{i+1}$, $M_i(v,u)\geq f(|V_{i+1}|)$, then the same is true for any telescoping of $(V,E)$. 
\end{rem}

An {\em ordered Bratteli diagram} is a Bratteli diagram together with a linear ordering on $r^{-1}(v)$ for each vertex $v\in V\setminus V_0$. This defines a partial order $\geq$ on $E$. An ordered Bratteli diagram $(V,E)$ together with a partial order $\geq$ on $E$ is denoted $(V,E,\geq)$. Let $E_{min}$ and $E_{max}$ denote the sets of minimal and maximal edges with respect to $\geq$, respectively. An {\em infinite path} in $(V,E)$ is a sequence of the form $(e_1,e_2,\cdots)$ where $e_i\in E_i$ and  $r(e_i)=s(e_{i+1})$ for all $i\in\N$. Such a path is called {\em minimal} (resp. {\em maximal}) if all its edges belong to $E_{min}$ (resp. $E_{max}$). An ordered Bratteli diagram is {\em properly ordered} if it is simple and the sets of minimal and maximal paths both are singletons. In this case we use the notation $E_{min}=x_{min}$ and $E_{max}=x_{max}$, respectively. 

Given a properly ordered Bratteli diagram $(V,E,\geq)$, it is possible to define a dynamic on it as follows. Let $X_B$ denote the space of all infinite paths on $E$ (where we assume that $X_B$ is infinite). We endow $X_B$ with a topology by giving a basis of clopen sets, namely the family of {\em cylinder} sets,
$$[e_1 , e_2 , \cdots, e_k ]_B = \{(f_1 , f_2 , \cdots) \in X_B: f_i=e_i \mbox{ for all } 1\leq i\leq k\}.$$
The space $X_B$ endowed with this topology is called the {\em Bratteli compactum} associated with $(V,E)$ and it is a Cantor space. We define the {\em Vershik map} $V_B$ on $X_B$ by setting $V_B(x_{max})=x_{min}$, and if $x=(e_1,e_2,\cdots)\neq x_{max}$, let $k$ be the smallest integer such that $e_k$ is not a maximal edge, let $f_k$ be the successor of $e_k$ on $E_k$, and define $V_B(x)=(f_1,f_2,\cdots, f_{k-1},f_k,e_{k+1},e_{k+2},\cdots)$, where $(f_1,\cdots, f_{k-1})$ is the minimal finite path on $E_1\circ E_2\circ \cdots \circ E_{k-1}$ with range equal to $s(f_k)$. The system $(X_B,V_B)$ is called the {\em Bratteli-Vershik system} associated to $(V,E,\geq)$.

Herman, Putnam and Skau showed that for any Cantor minimal system $(X,T)$, there exists a properly ordered Bratteli diagram $(V,E,\geq)$ such that $(X,T)$ and $(X_B,V_B)$ are conjugate \cite{HPS}. 

To any Bratteli diagram $(V,E)$ with sequence of incidence matrices $(M_i)_{i\in\N}$ one can associate an ordered group with unit as the inductive limit of the system
$$ \Z^{|V_0|}  \xrightarrow[]{M_0} \Z^{|V_1|}   \xrightarrow[]{M_1} \Z^{|V_2|}\xrightarrow[]{M_2} \cdots, $$
endowed with the induced order. The order unit corresponds to the element $[1,0]\in \varinjlim\limits_{i}(\Z^{|V_i|},M_i)$. This ordered group with unit is called the {\em dimension group} of $(V,E)$ and is denoted $K_0(V,E)$. Two Bratteli diagrams $(V,E)$ and $(V',E')$ have isomorphic dimension groups if and only if $(V',E')$ can be obtained from $(V,E)$ by a finite number of telescopings and microscopings. In particular, if there exists a Bratteli diagram $(\overline{V},\overline{E})$ such that $(V,E)$ is the telescoping of $(\overline{V},\overline{E})$ to the even levels and $(V',E')$ is the telescoping of $(\overline{V},\overline{E})$ to the odd levels (or {\em vice versa}), then $K_0(V,E)\cong K^0(V',E')$. On the other hand, $(V,E,\geq)$ is a properly ordered Bratteli diagram and $(X_B,V_B)$ is its associated Bratteli-Vershik system, then $K_0(V,E)\cong K^0(X_B,V_B)$ as ordered groups with unit \cite{HPS}.

\noindent Let $(V,E)$ be a simple Bratteli diagram with sequence of incidence matrices $(M_i)_{i\in \N}$. For each $i<j\in\N$, we denote $M_{[i,j)}=M_{j-1}M_{j-2}\cdots M_i$. For each $v\in V_j$, let $M_{[i,j)}(v)$ denote the sum $\sum_{u\in V_i}M_{[i,j)}(v,u)$ and let $\Vec{M}_{[i,j)}(v)$ denote the {\em incidence vector} of $v$, that is, $(M_{[i,j)}(v,u))_{u\in V_i}$. 
For each $i\geq 1$, let $\ell_i=\langle M_{[0,i)}\rangle=\min_{v\in V_i}\{M_{[0,i)}(v)\}$ and $L_i=\| M_{[0,i)}\|=\max_{v\in V_i}\{M_{[0,i)}(v)\}$.

For a positive integer matrix $M$, we say that $M$ has the {\em equal row sum (ERS)} property if the sum of the entries on a row of $M$ is constant. 

\begin{definition}\label{def:ersandpseudoers}
We say that a simple Bratteli diagram $(V,E)$ with sequence of incidence matrices $(M_i)_{i\in\N}$ is ERS if $\forall i\in\N, \ell_i=L_i$; we say that it is pseudo-ERS if $\lim_{i\to \infty}\frac{L_i}{\ell_i}=1$.
\end{definition}

\begin{rem}\label{rem:ers_epnp}
    For a simple Bratteli diagram, being ERS is equivalent to have the {\em equal path number property} (EPNP) introduced in \cite{Gjerde_Johansen_Bratteli_Toeplitz:2000}.
\end{rem}

\begin{definition}\label{def:pseudotoeplitz}
We say that a minimal subshift $(X,S)$ is pseudo-Toeplitz if it admits a Bratteli--Vershik representation $(V,E,\geq)$ such that $(V,E)$ is pseudo-ERS. 
\end{definition}

\Cref{def:pseudotoeplitz} is inspired by the following results relating ERS Bratteli diagrams and Toeplitz subshifts.

\begin{theorem}\cite[Theorem 8]{Gjerde_Johansen_Bratteli_Toeplitz:2000}\label{teo:erstoeplitz}
The family of expansive Bratteli-Vershik systems associated with ERS Bratteli diagrams coincides with the family of Toeplitz flows up to conjugacy.  
\end{theorem}

\begin{theorem}\cite[Theorem 1.2]{Sugisaki_toeplitz_bratteli_SOE:2001}
The family of Bratteli–Vershik systems associated with ERS Bratteli diagrams coincides with the family of Toeplitz flows up to strong orbit equivalence.
    
\end{theorem}

Let $\cA$, $\cB$ be two finite alphabets. Let $\tau \colon \cA^{\ast} \to \cB^{\ast}$ be a morphism. We say that $\tau$ is {\it non-erasing} if the image of any letter is a non-empty word. The {\it incidence matrix} of $\tau$ is the $|\cB|\times|\cA|$ matrix whose entry at a position $(b,a)$ is the number of times that $b$ appears in $\tau(a)$. For $a\in\cA$, the length of the word $\tau(a)\in \cB^{\ast}$ is denoted $|\tau(a)|$. The morphism $\tau$ is {\em positive} if the entries of $|\cB|\times|\cA|$ are all positive and is {\it left proper} (resp. {\it right proper}) if there exists a letter $b\in\cB$ such that for every $a\in\cA$, $\tau(a)$ starts with $b$ (resp. ends with $b$); it is {\it proper} if it is both left and right proper. We say that $\tau$ is a {\it hat morphism} if for all $a,b\in \cA$, the letters appearing in $\tau(a)$ and $\tau(b)$ are all distinct. By concatenation, a morphism $\tau\colon \cA^{\ast}\to \cB^{\ast}$ can be extended to $\cA^\N$ and $\cA^\Z$.\\

A {\it directive sequence} ${\bt}=(\tau_i)_{i\geq 0}$ is a sequence of non-erasing morphisms $\tau_i\colon \cA_{i+1}^{\ast}\to \cA_i^{\ast}$, $i\in \N$. We let $\tau_{[i,k)}$ denote the composition $\tau_i\circ\tau_{i+1}\circ\cdots\circ \tau_{k-1}$. We say that $\bt$ is {\it everywhere growing} if $\langle \tau_{[0,i)} \rangle=\min_{a\in \cA_i}\{|\tau_{[0,i)}(a)|\}$ tends to $\infty$ as $i\to\infty$. We say that $\bt$ is {\it primitive} if for every $i\geq 0$, there exists $k\geq i$ such that $\tau_{[i,k)}$ has a positive incidence matrix.

For $i\geq 0$, the {\it language of order $i$} $L_{{\bt}}^{(i)}$ associated with ${\bt}$ is defined as
$$L_{{\bt}}^{(i)}=\{w\in\cA_i^{\ast}: \exists k>i, \exists a\in \cA_k, w\prec \tau_{[i,k)}(a)\}.$$
For each $i\geq 0$, the set $X_{{\bt}}^{(i)}$ is the set of infinite words $x\in\cA_i^\Z$  whose factors belong to $L_{{\bt}}^{(i)}$. 
We set $X_{{\bt}}=X_{{\bt}}^{(0)}$, $L_{{\bt}}=L_{{\bt}}^{(0)}$ and call $(X_{{\bt}}, S)$ the {\it $\cS$-adic system generated by the directive sequence} ${\bt}$, where $S$ is the shift transformation. For all $\ell \geq 1$, we denote by $L_{{\bt},\ell}^{(i)}$ the subset of length $\ell$ factors of $L_{{\bt}}^{(i)}$. When the directive sequence $\bt$ is primitive, $(X_{\bt},S)$ is a minimal subshift.

Given an ordered Bratteli diagram $(V,E,\geq)$, let $i\geq 1$ and consider $V_i$, $V_{i+1}$ as finite alphabets. For every $u\in V_{i+1}$, consider the ordered list $(e_1,e_2,\cdots, e_k)$ of edges on $E_{i+1}$ arriving to $u$, and let $(v_1,v_2,\cdots, v_k)$ be the list of labels of the sources of these edges in $V_i$. This defines a morphism $\tau_i\colon V_{i+1}^{\ast} \to V_{i}^{\ast}$,  $u\mapsto v_1v_2\cdots v_k$. For $i=0$, we let $\tau_0\colon V_1^{\ast} \to E_1^{\ast}$ denote the morphism such that  
$\tau_0(v)=e_1(v)\ldots e_\ell(v)$, where  $e_1(v),\ldots,e_\ell(v)$ are the edges connecting $V_0$ with $v$ according to the order of the diagram. For $i\geq 0$ we say that $\tau_i$ is {\it the morphism read on $(V,E,\geq)$ at level $i$} and that the directive sequence ${\boldsymbol \tau}=(\tau_i)_{i\geq 0} $ is the {\it sequence of morphisms read  on} $(V,E,\geq)$. 

\begin{rem}\label{rem:hat}
    Note that the morphism read on $(V,E,\geq)$ at level $0$ is a hat morphism.
\end{rem}

The notion of the morphism read on an ordered Bratteli diagram gives an $\cS$-adic subshift naturally associated with each Bratteli-Vershik system. The following result gives sufficient conditions for these two representations to be conjugate as dynamical systems. It corresponds to a slight modification of \cite[Proposition 2.2]{Durand&Leroy:2012}.

\begin{proposition}\cite[Proposition 2.2]{Durand&Leroy:2012}\label{prop:isom}
Let $(X_{\bt},S)$ be the minimal $\cS$-adic subshift defined by the directive sequence $\bt  = (\tau_i)_{i\geq 0}$, where $\tau_0$ is a hat morphism and $\tau_i$ is proper for all $i\geq 1$. Suppose that all morphisms $\tau_i$ extend by concatenation to a one-to-one map from $X^{(i+1)}_{\bt}$ to $X^{(i)}_{\bt}$. Then, $(X_{\bt},S)$ is conjugate to the Bratteli-Vershik system $(X_{B} , V_{B})$ associated with the ordered Bratteli diagram $B=(V,E,\geq)$, where  $B$ is such that $\bt$ corresponds to the sequence of morphisms read  on $(V,E,\geq)$.
\end{proposition}

Let $(X_{\bt},S)$ be the minimal $\cS$-adic subshift defined by the directive sequence $\bt  = (\tau_i)_{i\geq 0}$. For $i\geq 1$ and $x\in X^{(i+1)}_{\boldsymbol{\tau}}$, we define the set $C_{\tau_i}(x)$ of {\it cutting points} of $\tau_i(x)$ as the following set,
$$C_{\tau_i}(x)=\{|\tau_i(x_{[0,\ell)})|:\ell >0\}\cup\{0\}\cup \{-|\tau_i(x_{[\ell,0)})|:\ell <0\}.$$ 

\begin{lemma}\label{lem:recognizable}
    Let $(X_{\bt},S)$ be the minimal $\cS$-adic subshift defined by a directive sequence $\boldsymbol{\tau}  = (\tau_i:\cA_{i+1}^\star \to \cA_i^\star)_{i\geq 0}$. Suppose that for all $i\geq 1$, 
    \begin{enumerate}
        \item $\tau_i$ is injective on letters, and
        \item $\forall x, y \in X_{\bt}^{(i+1)}$, $\tau_i(x)=\tau_i(y)\implies C_{\tau_i}(x)=C_{\tau_i}(y)$.
    \end{enumerate}
    Then, each $\tau_i$ extends by concatenation to a one-to-one map from $X^{(i+1)}_{\bt}$ to $X^{(i)}_{\bt}$.
\end{lemma}

\begin{proof}
    Let $x,y\in X^{(i+1)}_{\boldsymbol{\tau}}$ be two infinite words such that $\tau_i(x)=\tau_i(y)$. We want to prove that $x=y$. By (2), the cutting points of $x$ and $y$ are located exactly in the same places, so  $\tau_i(x_\ell)=\tau_i(y_\ell)$ for all $\ell\in\Z$. Since $\tau_i$ is injective on $\cA_{i+1}$ by (1), we conclude that $x_\ell=y_\ell$ for all $\ell\in\Z$, thus $x=y$.
\end{proof}

Note that under the assumptions of \cref{lem:recognizable}, if the directive sequence is proper, then each morphism $\tau_i$ is injective on words.

Combining \Cref{rem:hat}, \Cref{prop:isom} and \Cref{lem:recognizable} we obtain the following corollary.

\begin{corollary}\label{cor:suffconditions_isom}
    Let $B=(V,E,\geq)$ be a simple ordered Bratteli diagram, let  $\bt=(\tau_i\colon V_{i+1}^{\ast} \to V_{i}^{\ast})_{i\in\N}$ be the sequence of morphisms read on $B$. If for each $i\geq 1$, $\tau_i$ is proper and $\bt$ satisfies the hypotheses of \Cref{lem:recognizable}, then $(X_B, V_B)$ is conjugate to the minimal $\cS$-adic subshift $(X_{\bt},S)$. 
\end{corollary}

\begin{theorem}[{\cite{Boyle_Handelman_entropy_vs_oe:1994}. See also \cite{Durand_combinatorics_bratelli:2010} or \cite[Theorem 4.3]{Berthe_Delecroix_beyond_substitutive_Sadic:2014}}]\label{theo:sadicentropy}
Let $(X_{\bt},S)$ be the $\cS$-adic shift defined by the directive sequence $\boldsymbol{\tau}  = (\tau_i:\cA_{i+1}^\star \to \cA_i^\star)_{i\geq 0}$. Then, the topological entropy of $(X_{\bt},S)$ satisfies,
$$h(X_{\bt},S)\leq \inf_{i\in\N}\frac{\log(|\cA_i|)}{\langle \tau_{[0,i)} \rangle}.$$

\end{theorem} 

\section{The splitting procedure}\label{sec:splitting}

We describe here one of the main tools we use to modify Bratteli diagrams in order to create entropy while keeping the dimension group and the pseudo-ERS property. It is inspired by the ideas of \cite{Sugisaki_subshift_within_SOE:2007}. This procedure is then used in the proofs of \Cref{theo:low} and \Cref{theo:main}. 

Let $(V,E)$ be a simple Bratteli diagram and let $(t_i)_{i\in\N}$ be a sequence of telescoping depths. Let $(\widetilde{V,}\widetilde{E})$ be the telescoping of $(V,E)$ to $(t_i)_{i\in\N}$. 

Let $i\geq 1$. Chose any vertex $\bar{u}\in\widetilde{V}_i$ and fix it. To each vertex $u\in\widetilde{V}_i$ we associate a number $q_i^{(u)}\in\N^\star$ and a set $H_i^{(u)}$ of $q_i^{(u)}$ ``new'' vertices. Let $V'_i=\bigcup_{u\in\widetilde{V}_i}H_i^{(u)}$. The set $V_i'$ should be seen as an intermediate set of vertices between $\widetilde{V}_i$ and $\widetilde{V}_{i+1}$, that is, we split each vertex in $\widetilde{V}_i$ into $q_i^{(u)}$ new vertices. Let $v_{\min}^i, v_{\max}^i$ be two different vertices belonging to $H_i^{(\bar{u})}$. We list the vertices of $\widetilde{V}_{i+1}$ as
$$\widetilde{V}_{i+1}=\{x_1,x_2,\ldots , x_{|\widetilde{V}_{i+1}|}\}.$$
Suppose that for each $u\in \widetilde{V}_i$ and $x=x_j\in \widetilde{V}_{i+1}$ one has
\begin{equation}\label{eq:conditionq_i}
    \widetilde{M}_i(x_j,u)\geq q_i^{(u)}+j-1.
\end{equation}
For $u\in\widetilde{V}_i$ and $x\in\widetilde{V}_{i+1}$, if $u\neq \bar{u}$, let $k_{x,u}$, $r_{x,u}$ be the unique integers such that
$$\widetilde{M}_i(x,u)=k_{x,u}q_i^{(u)}+r_{x,u}$$
and $0\leq r_{x,u}<q_i^{(u)}$. For $\bar{u}$, let $k_{x,\bar{u}}$ and $r_{x,\bar{u}}$ be the unique integers such that
$$\widetilde{M}_i(x,\bar{u})-j-1=k_{x,\bar{u}}(q_i^{(\bar{u})}-2)+r_{x,\bar{u}}$$
and $0\leq r_{x,\bar{u}}<q_i^{(\bar{u})}-2$, where $1\leq j\leq |\widetilde{V}_{i+1}|$ and $x=x_j$.

We define two non-negative matrices, $A_i$ and $B_i$, connecting vertices in $V_i'$ with vertices in $\widetilde{V}_i$ and vertices in $\widetilde{V}_{i+1}$ with vertices in $V_i'$, respectively, in such a way that $\widetilde{M}_i=B_iA_i$ for all $i\geq 1$. For all $v\in V_i'$ and $u\in \widetilde{V}_i$, 

$$
A_i(v,u)= \begin{cases} 1 & \text{ if } v \in H_i^{(u)}\\
0 & \text{ else}. 
\end{cases}
$$

For $1\leq j\leq |\widetilde{V}_{i+1}|$, define $B_i(x,v_{\min}^i)=j$ and $B_i(x,v_{\max}^i)=1$ for $x=x_j$.  Note that this choice is possible thanks to Condition \eqref{eq:conditionq_i}.

For the rest of the vertices $v\in V_i'$, we define $B_i(x_j,v)$ as follows. For $v\in H_i^{(u)}$, $u\neq \bar{u}$, choose $B_i(x,v)\in \{k_{x,u},k_{x,u}+1\}$ in such a way that
     $$\sum_{v\in H_i^{(u)}}B_i(x,v)=\widetilde{M}_i(x,u).$$
     This means that we assign $k_{x,u}$ edges to $q_i-r_{x,u}$ vertices in $H_i^{(u)}$ and $k_{x,u}+1$ edges to the $r_{x,u}$ remaining vertices in $H_i^{(u)}$. For $v\in H_i^{(\bar{u})}\setminus\{v_{\min}^i, v_{\max}^i\}$, choose $B_i(x,v)\in \{k_{x,\bar{u}},k_{x,\bar{u}}+1\}$ in such a way that
     $$\sum_{v\in H_i^{(\bar{u})}\setminus\{v_{\min}^i,v_{\max}^i\}}B_i(x,v)=\widetilde{M}_i(x,\bar{u})-j-1.$$
     
\begin{rem}\label{rem:distinctincidencevector}
Note that, by construction of the matrix $B_i$, if $x,x'\in \widetilde{V}_{i+1}$ are different vertices, then $\Vec{B_i}(x)\neq \Vec{B_i}(x')$, that is, $x$ and $x'$ have distinct $B_i$-incidence vectors, since at least $B_i(x,v_{\min}^i)\neq B_i(x',v_{\min}^i)$.
\end{rem}

Let $V_0'=\widetilde{V}_0=V_0$. Let $M_0'=A_1\widetilde{M}_0$ and for each $i\geq 1$, let $M_i'=A_{i+1}B_i$. Let $(V',E')$ be the simple Bratteli diagram with vertices $(V_i')_{i\in\N}$ and incidence matrices $(M_i')_{i\in\N}$. By construction, $K_0(V',E')\cong K_0(\widetilde{V}, \widetilde{E})\cong K_0(V,E)$. We call $(V',E')$ the {\em splitting of $(V,E)$ with respect to $(t_i)_{i\in\N}$ and $(q_i^{(u)})_{u\in \widetilde{V}_i, i\geq 1}$}. In \Cref{theo:low} and \Cref{theo:main} we appropriately choose the parameters $(t_i)_{i\in\N}$ and $(q_i^{(u)})_{i\geq 1,u\in \widetilde{V}_i}$ so that Condition \eqref{eq:conditionq_i} and other supplementary conditions hold.\\

\begin{rem}\label{rem:preservesers2}
Note that, thanks to the definition of the matrix $A_i$, for any $u\in \widetilde{V}_i$ and $v\in H_i^{(u)}$, $\ell_u=\ell_v$. Thus, in particular, $L_i'=\widetilde{L}_i=L_{t_i}$, $\ell_i'=\widetilde{\ell}_i=\ell_{t_i}$, which implies that the passage from $(\widetilde{V},\widetilde{E})$ to $(V',E')$ does not modify the pseudo-ERS property. Since the telescoping procedure does not modify it either, we have that if $(V,E)$ is pseudo-ERS, then so it is $(V',E')$. By the same argument, if $(V,E)$ is ERS, so it is $(V',E')$.
\end{rem}

Let $(V,E)$ be a simple Bratteli diagram and let $(V',E')$ be the splitting of $(V,E)$ with respect to $(t_i)_{i\in\N}$ and $(q_i^{(u)})_{u\in\widetilde{V}_i,i\geq 1}$. Let $i\geq 1$ and consider a vertex $x=x_j\in \widetilde{V}_{i+1}$, $1\leq j\leq |\widetilde{V}_{i+1}|$. Define

\begin{equation}\label{eq:o(x)}
\cO(x)=\frac{\left(\left(\sum_{u\in\widetilde{V_i}}\widetilde{M_i}(x,u)\right)-j-1\right)!}{\prod_{v\in V_i'\setminus
    \{v_{\min}^i,v_{\max}^i\}}B_i(x,v)!}.
\end{equation}

\begin{lemma}\label{lem:recognizable2}
    Suppose that for each $x\in \widetilde{V}_{i+1}$, $\cO(x)\geq q_{i+1}^{(x)}$. Then, it is possible to put an order $\geq$ on the edges of $(V',E')$ such that the sequence $\bt=(\tau_i:(V_{i+1}')^\star  \to (V_i')^\star)_{i\in\N}$ of read morphisms on $(V',E')$ satisfies the hypotheses of \cref{cor:suffconditions_isom}.
\end{lemma}

\begin{proof}
    For any vertex $v\in V_{i+1}'$ such that $v\in H_{i+1}^{(x)}$ and $x=x_j$, with $1\leq j\leq |\widetilde{V}_{i+1}|$, let $\tau_i(v)$ begin with the word $(v_{\min}^i)^j$ and end with the letter $v_{\max}^i$. The intermediate portion of $\tau_i(v)$ will be determined later. With this choice, the image of any vertex $v\in V_{i+1}'$ under $\tau_i$ begins with $v_{\min}^i$ and ends with $v_{\max}^i$, so $\tau_i$ is proper. On the other hand, the cutting points of any $\tau_i(v)$ are located exactly in those places where a letter $v_{\max}^i$ is followed by a letter $v_{\min}^i$, so if $\tau_i(v)=\tau_i(v')$ then $C_{\tau_i(v)}=C_{\tau_i}(v')$. Thus, to prove that $\bt$ satisfies the hypotheses of \Cref{cor:suffconditions_isom}, it suffices to prove that it is possible to choose $\tau_i$ injective on $V'_{i+1}$. \\
    Note that, independently of the order assigned to the edges in $E_i'$, if $v_1, v_2\in V_{i+1}'$ verify $\Vec{M_i'}(v_1)\neq \Vec{M_i'}(v_2)$, that is, if $v_1$ and $v_2$ have distinct final incidence vectors, then the read morphism $\tau_i$ satisfies $\tau_i(v_1)\neq \tau_i(v_2)$. From the definition of the matrices $A_{i+1}$ and $B_i$, we deduce that if $v_1\in H_{i+1}^{(x_1)}$ and $v_2\in H_{i+1}^{(x_2)}$ with $x_1\neq x_2$, then $\tau_i(v_1)\neq \tau_i(v_2)$ (see \Cref{rem:distinctincidencevector}). Thus, to prove that it is possible to choose $\tau_i$ injective on letters, it is enough to show that, for any $x\in \widetilde{V}_{i+1}$, the number of orders that can be assigned to the edges connecting a vertex $v\in H_{i+1}^{(x)}$ with the set of vertices $V_i'\setminus \{v_{\min}^i,v_{\max}^j\}$ (those edges for which the order has not been established yet), is at least $q_{i+1}^{(x)}$. We conclude by noticing that $\cO(x)$ corresponds precisely to the maximum number of orders which is possible to assign to the edges connecting a vertex in $H_{i+1}^{(x)}$ with the vertices in $V_i'\setminus \{v_{\min}^i,v_{\max}^i\}$.
\end{proof}

A simplified version of the splitting procedure described above can be performed to assume that, up to SOE, the number of vertices in $(V,E)$ goes to infinity as the level increases, a fact we use in the proofs of \Cref{theo:low} and \Cref{theo:main}. We briefly describe this simplified procedure. This time we define the sequences $(t_i)_{i\in\N}$ and $(q_i^{(u)})_{i\geq 1, u\in \widetilde{V}_i}$ inductively. 

Let $t_0=0$ and $V_0'=V_0$, let $t_1=1$. For $i\geq 0$, suppose we have defined $t_i$. Chose and fix a vertex $u^*\in \widetilde{V}_i=V_{t_i}$. For each $u\in \widetilde{V}_i$, let
    $$q_i^{(u)}=\begin{cases}
        i & \text{ if } u=u^*\\
        1 & \text{ if } u\neq u^*
    \end{cases},$$
    and let $H_i^{(u)}$ be a set of size $q_i^{(u)}$, as before. Let $V_i'=\bigcup_{u\in \widetilde{V}_i} H_i^{(u)}$. Let $t_{i+1}>t_i$ be chosen large enough so that for each $x\in \widetilde{V}_{i+1}=V_{t_{i+1}}$,
    $$\widetilde{M}_i(x,u^*)\geq q_i^{u^*}.$$
    For $x\in \widetilde{V}_{i+1}$, we perform the Euclidean division of $M_i(x,u^*)$ by $q_i^{(u^*)}$: let $k_x$ and $r_{x}$ be the unique integers such that
    $$\widetilde{M}_i(x,u^*)=k_x q_i^{(u^*)}+r_x.$$
    Let $A_i$ be the $|V_i'|\times |\widetilde{V}_i|$ matrix defined as follows. For $u\in \widetilde{V}_i$ and $v\in V_i'$,
    $$A_i(v,u)=\begin{cases}
        1 & \text{ if } v\in H_i^{(u)}\\
        0 & \text{ else.}
    \end{cases}$$
    Let $B_i$ be the $|\widetilde{V}_i|\times |V_i'|$ matrix defined as follows. For $x\in \widetilde{V}_{i+1}$ and $v\in V_i'$, if $v\in H_i^{(u^*)}$, choose $B_i(x,v)\in \{k_x,k_x+1\}$ in such a way that
    $$\sum_{v\in H_i^{(u^*)}}B_i(x,v)=\widetilde{M}_i(x,u^*).$$
    If $v\not\in H_i^{(u^*)}$, set $B_i(x,v)=\widetilde{M}_i(x,u)$, where $u\neq u^*$ is such that $\{v\}= H_i^{(u)}$.\\
    Let $M_0'=A_1\widetilde{M}_0$ and for each $i\geq 1$, let $M_i'=A_{i+1}B_i$. Let $(V', E')$ be the simple Bratteli diagram with sequence of vertices $(V_i')_{i\in\N}$ and sequence of incidence matrices $(M_i')_{i\in\N}$. By construction, $K_0(V,E)\cong K_0(V',E')$ and for each $i\geq 1$, $m_i:=|V_i'|\geq i$. Thus, in $(V',E')$ the number of vertices on each level goes to infinity. We have proved the following.

    \begin{proposition}\label{prop:unbounded}
        For each simple Bratteli diagram $(V,E)$ there exists a simple Bratteli diagram $(V',E')$ such that $K_0(V,E)\cong K_0(V',E')$ and the sequence of number of vertices $(|V'_i|)_{i\in\N}$ goes to infinity.
    \end{proposition}

\section{Proof of the main results}\label{sec:proofs}

We start by proving a combinatorial lemma that will be used in the proofs of \Cref{theo:low} and \Cref{theo:main}.

\begin{lemma}\label{lem:stirling2}
Let $d\in\N$ and $m_1,\ldots, m_d$ be positive integers, $m_u\geq 2$. For each $1\leq u\leq d$, let $a_{1u}, a_{2u}, \ldots , a_{m_uu}$ be positive integers, and define $N_u=\sum_{v=1}^{m_u}a_{vu}$, $A=\sum_{u=1}^d N_u$. Suppose that for all $1\leq u\leq d$, $m_u-1\leq N_u/a_{vu}$. 
Then, provided the $a_{vu}$'s are large enough with respect to $m_u,d$, one has
$$\log\left(\frac{A!}{\prod_{u=1}^d\prod_{v=1}^{m_u}a_{vu}!}\right)\geq \sum_{u=1}^d N_u(\log(A/N_u)+\log(m_u-1))+\varepsilon(A),$$
where $\varepsilon(A)=O_{d,M}(\log(A))$ and $M=\max_{1\leq u\leq d}\{m_u\}$.
\end{lemma}
\begin{proof}
    According to the Stirling formula, 
    $$\log(A!)=A\log(A)-A+O(\log(A))$$
    and 
    $$\log(a_{vu}!)=a_{vu}\log(a_{vu})-a_{vu}+O(\log(a_{vu})).$$
    Since $d$ and the $m_u$'s are fixed, there exists $c>0$ such that
    $$-cd\max_{1\leq u\leq d}\{m_u\}O(\log(A))\leq \sum_{u=1}^d\sum_{v=1}^{m_u}O(\log(a_{vu}))\leq cd\max_{1\leq u\leq d}\{m_u\}O(\log(A)),$$
    which means that $\sum_{u=1}^d\sum_{v=1}^{m_u}O(\log(a_{vu}))=O_{d,M}(\log(A))$. Thus,

    $$\log\left(\frac{A!}{\prod_{u=1}^d\prod_{v=1}^{m_u}a_{vu}!}\right)=A\log(A)-A-\sum_{u=1}^d\sum_{v=1}^{m_u}(a_{vu}\log(a_{vu})-a_{vu})+O_{d,M}(\log(A)).$$
  Since $A=\sum_{u=1}^dN_u=\sum_{u=1}^d\sum_{v=1}^{m_u}a_{vu}$, we obtain
  \begin{align*}
  \log\left(\frac{A!}{\prod_{u=1}^d\prod_{v=1}^{m_u}a_{vu}!}\right) &=    
  \sum_{u=1}^d\sum_{v=1}^{m_u}a_{vu}(\log(A)-\log(a_{vu}))+O_{d,M}(\log(A))\\
  &=\sum_{u=1}^d\sum_{v=1}^{m_u}a_{vu}(\log(A/N_u)+\log(N_u/a_{vu}))+O_{d,M}(\log(A)).
  \end{align*}
      
  Since $N_u/a_{vu}\geq m_u-1$, we get
  $$\log\left(\frac{A!}{\prod_{u=1}^d\prod_{v=1}^{m_u}a_{vu}!}\right)\geq\sum_{u=1}^dN_u(\log(A/N_u)+\log(m_u-1))+O_{d,M}(\log(A)).$$
    
\end{proof}

\begin{proof}[Proof of \cref{theo:pseudoers}]

We will inductively define a sequence of telescoping depths $(t_i)_{i\in\N}$ and a sequence of intermediate levels to construct the diagram $(V',E')$. 
Let $(M_i)_{i\in\N}$ be the sequence of incidence matrices of $(V,E)$.\\
For $i\in\{0,1\}$, let $t_i=i$. Let also $V_0'=V_0$.\\
Let $i\geq 1$. Suppose we have defined $t_i$. We will define $t_{i+1}$ and $V_{i}'$. Let $\widetilde{V}_i=V_{t_i}$. For each $u\in \widetilde{V}_i$, let $\ell_u=M_{[0,t_i)}(u)$, that is, the number of paths connecting $u$ to the top vertex, and denote $\widetilde{L}_i=L_{t_i}$. Define $K_i>(f(2|\widetilde{V}_i|)+1)\widetilde{L}_i$ large enough so that
\begin{equation}\label{eq:conditionK_i}
    (K_i+\widetilde{L}_i)h(K_i+\widetilde{L}_i)\geq 2\widetilde{L}_i
\end{equation}

Such a $K_i$ always exists since $ih(i)\nearrow +\infty$ and $, \widetilde{V}_i$, $\widetilde{L}_i$ only depends on the parameters of level $i$.\\
For each $u\in \widetilde{V}_i$ define $K_i^{(u)}=\lfloor\frac{K_i}{\ell_u}\rfloor$. From the B\'ezout Theorem, it follows that if $p,q\in\N^\star$ are coprime and $M\geq pq+1$, there exist $n,m\in\N^\star$ such that $M=np+mq$. Thus, there exists $t^\star>t_i$ such that for all $t>t^\star$, $\min_{u\in \widetilde{V}_i,v\in V_t}\{M_{[t_i,t)}(v,u)\}>K_i$ and for each $u\in \widetilde{V}_i, v\in V_t$, $\exists n_{u,v}, m_{u,v}\in\N^\star$ such that 
$$M_{[t_i,t)}(v,u)=n_{u,v}K_i^{(u)}+m_{u,v}(K_i^{(u)}+1).$$
Define $t_{i+1}=t^\star$ and $\widetilde{V}_{i+1}=V_{t_{i+1}}$. For each $u\in\widetilde{V}_i$, let $H_i^{(u)}=\{w_u^1, w_u^2\}$ be a set of two vertices, let $V_i'=\bigcup_{u\in\widetilde{V}_u}H_i^{(u)}$. Note that $|V_i|=2|\widetilde{V}_i|$. The set $V_i'$ is an intermediate set of vertices between $\widetilde{V}_i$ and $\widetilde{V}_{i+1}$.\\
For each $i\in\N$, we define two matrices, $A_i$ and $B_i$, connecting vertices in $V_i'$ with those in $\widetilde{V}_i$ and vertices in $\widetilde{V}_{i+1}$ with those in $V_i'$, respectively, as follows. For each $u\in\widetilde{V}_i, w\in V_i'$,
$$
A_i(w,u)= \begin{cases} K_i^{(u)} & \text{ if } w=w_u^1 \\
(K_i^{(u)}+1) & \text{ if } w=w_u^2\\
0 & \text{ else.}
\end{cases}
$$
For each $w\in V'_i, v\in \widetilde{V}_{i+1}$,
$$
B_i(v,w)= \begin{cases} n_{u,v} & \text{ if } w=w_u^1\text{ for some } u\in \widetilde{V}_i \\
m_{u,v} & \text{ if } w=w_u^2 \text{ for some} u\in \widetilde{V}_i. 
\end{cases}
$$

Let $(M_i')_{i\geq 0}$ be the sequence of incidence matrices defined by $M_0'=A_1M_0$ and $M_i'=A_{i+1}B_i$ for all $i\geq 1$. Let $(V',E')$ be the diagram given by the matrices $(M_i')_{i\in\N}$. By construction, $M_i=B_iA_i$ for all $i\in\N$, which implies that $K_0(V,E)\cong K_0(V',E')$. This proves the first item of the statement.\\
We now prove that for all $i\geq 1$, $r\in V'_{i-1}$ and $s\in V'_i$, $M'_{i-1}(s,r)\geq f(|V'_i|)$. Note that, by construction, 
$$M_{i-1}(s,r)\in\{K_i^{(v)}n_{u,v},(K_i^{(v)}+1)n_{u,v},K_i^{(v)}m_{u,v},(K_i^{(v)}+1)m_{u,v}\},$$ where $u\in \widetilde{V}_{i-1}$ and $v\in \widetilde{V}_i$ are such $r\in H_{i-1}^{(u)}$ and $s\in H_i^{(v)}$. Since $n_{u,v}, m_{u,v}$ are positive, this implies that
$$M'_{i-1}(s,r)\geq K_i^{(v)}\geq \frac{K_i}{\ell_v}-1\geq \frac{K_i-\widetilde{L}_i}{\widetilde{L}_i}.$$
Since $K_i>(f(2|\widetilde{V}_i|)+1)\widetilde{L}_i$, we get that
$$M'_{i-1}(s,r)\geq\frac{f(2|\widetilde{V}_i|)+1)\widetilde{L}_i-\widetilde{L}_i}{\widetilde{L}_i}=f(2|\widetilde{V}_i|)=f(|V'_i|).$$
This proves the second item of the statement.\\
Finally, we prove that $(V',E')$ is pseudo-ERS with the prescribed speed of convergence of $\frac{\ell'_i}{L'_i}$ towards $1$. Let $i\geq 1$, let $w\in V'_i$ such that $\ell'_i=\ell_w$ and let $u\in \widetilde{V}_i$ such that $w\in H_i^{(u)}$. Then, we have that $\ell'_i$ belongs to $\{K_i^{(u)}\ell_u,(K_i^{(u)}+1)\ell_u\}$. In particular,
$$\ell'_i\geq K_i^{(u)}\ell_u\geq \left(\frac{K_i}{\ell_u}-1\right)\ell_u=K_i-\ell_u\geq K_i-\widetilde{L}_i.$$
On the other hand, if $x\in V'_i$ is such that $L'_i=\ell_x$ and $v\in\widetilde{V}_i$ is such that $x\in H_i^{(v)}$, we have that $L'_i$ belongs to $\{K_i^{(v)}\ell_v,(K_i^{(v)}+1)\ell_v\}$. In particular,
$$L'_i\leq (K_i^{(v)}+1)\ell_v\leq \left(\frac{K_i}{\ell_v}+1\right)\ell_v=K_i+\ell_v\leq K_i+\widetilde{L}_i.$$
We thus get that
$$\frac{\ell'_i}{L'_i}\geq \frac{K_i-\widetilde{L}_i}{K_i+\widetilde{L}_i}=1-\frac{2\widetilde{L}_i}{K_i+\widetilde{L}_i}.$$
By \Cref{eq:conditionK_i}, we obtain that 
$$\frac{\ell'_i}{L'_i}\geq 1-h(K_i+\widetilde{L}_i).$$
Since $L'_i\leq K_i+\widetilde{L}_i$ and $h$ is decreasing, $h(K_i+\widetilde{L}_i)\leq h(L'_i)$ and thus we get that $\frac{\ell'_i}{L'_i}\geq 1-h(L_i')$. This proves the third item of the statement and concludes the proof of the theorem.\\

\end{proof}

Before giving the proof of \Cref{theo:low} we prove the following preparatory lemma.

\begin{lemma}\label{lem:m_i}
    Let $(V,E)$ be a simple Bratteli diagram, let $F\colon \N^\star \to\N^\star$ be a function satisfying $\lim_{i\to\infty}F(i)\nearrow+\infty$ and $\lim_{i\to\infty}\frac{F(i)}{i}=0$, and let $f\colon \N^\star\to\R_+$ be any increasing function. There exists a pseudo-ERS Bratteli diagram $(V',E')$ with sequence of incidence matrices $(M'_i)_{i\in\N}$ such that the following conditions hold,
    
 \begin{enumerate}
    \item $K_0(V,E)\cong K_0(V',E')$,
    \item for all $i\geq 1$, $u\in V'_{i-1}$ and $v\in V'_i$,
    $M'_{i-1}(v,u)\geq f(|V'_i|)$,
    \item for all $i\geq 1$, $\frac{\ell'_i}{L'_i}\geq 1-\frac{1}{F(L_i')}$.
    \end{enumerate}
    Moreover, the sequence of positive integers $(m_i)_{i\in\N}$ defined by $m_i=F(L_i')$ verifies that $m_iL'_i\leq (m_i+1)\ell'_i$ for all $i\in\N$.
\end{lemma}
\begin{proof}
    Applying \Cref{theo:pseudoers} to $(V,E)$ with the function $h(i)=\frac{1}{F(i)+1}$, we obtain immediately items (1) and (2). Note that we can apply \Cref{theo:pseudoers} in this case since $h(i)$ goes to $0$ and $ih(i)$ diverges to $+\infty$ as $i$ goes to $\infty$. The third item in \Cref{theo:pseudoers} implies that for all $i\geq 1$,
    $$\frac{\ell'_i}{L'_i}\geq 1-\frac{1}{F(L'_i)+1}\geq 1-\frac{1}{F(L'_i)}.$$
    This proves item (3). Note that the inequality $\frac{\ell'_i}{L'_i}\geq 1-\frac{1}{F(L'_i)+1}$ is rewritten as
    $$\frac{\ell'_i}{L'_i}\geq\frac{F(L'_i)+1-1}{F(L'_i)+1}=\frac{m_i}{m_i+1},$$
    which implies that $m_iL'_i\leq (m_i+1)\ell'_i$. This concludes the proof of the lemma.
\end{proof}

\begin{proof}[Proof of \cref{theo:low}]
Let $(X,T)$ be a minimal Cantor system and let $(V,E)$ be a simple Bratteli diagram such that $K_0(V,E)\cong K^0(X,T)$. Suppose the hypotheses of the statement are satisfied. 
Thanks to\cref{prop:unbounded} we may assume that $(|V_i|)_{i\in\N}$ goes to $\infty$ as $i$ increases, and thanks to \cref{lem:m_i} we may also assume that $(V,E)$ is pseudo-ERS and that for any increasing function $f$, the number of paths from the top vertex to any vertex at level $V_i$ is greater than $f(|V_i|)$. Moreover, applying \cref{lem:m_i} to the function $F$ which satisfies \cref{eq:suffcondlow}, we know that the sequence $m_i=F(L_i)$ will verify $m_iL_i\leq (m_i+1)\ell_i$. Choose any increasing function $\widetilde{f}\colon \R_+\to\R_+$ which grows fast enough so that 

\begin{equation}\label{eq:condition_f}
    \widetilde{f}^{-1}(n)\leq \frac{n}{F(n)+2}
\end{equation}
for all $n$ large. Let $f$ be the restriction of $\widetilde{f}$ to $\N^\star$. This technical condition will be used to prove that $\liminf_{n\infty}\frac{p_Y(n)}{g_n}=0$.\\

For a given sequence of telescoping depths $(t_i)_{i\in\N}$, $(\widetilde{V},\widetilde{E})$ the telescoping of $(V,E)$ to $(t_i)_{i\in\N}$ and $i\geq 1$, choose and fix a vertex $\overline{u}\in\widetilde{V}_i$ as in \Cref{sec:splitting}. Recall that for each $u\in\widetilde{V}_i$, $\ell_u$ denotes the number of paths from the top vertex to $u$. For $u\in\widetilde{V}_i$ and $x=x_j\in \widetilde{V}_{i+1}$ with $1\leq j\leq |\widetilde{V}_{i+1}|$, define 
$$N_{x,u}= \begin{cases}
   \widetilde{M}_i(x,u) & \text{ if } u\neq \overline{u}\\
   \widetilde{M}_i(x,u)-j-1 & \text{ if } u=\overline{u}
\end{cases},$$
Define also $A_x=\sum_{u\in\widetilde{V}_i}N_{x,u}$.\\

We now inductively define the appropriate sequences $(t_i)_{i\in\N}$ and $(q_i^{(u)})_{i\geq 1, u\in\widetilde{V}_i}$ and then apply the splitting procedure described in \cref{sec:splitting}. 

Let $t_0=0, t_1=1$. Let $i\geq 1$ and suppose we have defined $t_{i}$, thus $\widetilde{V}_i$ is defined. For each $u\in\widetilde{V}_i$, define
\begin{equation}\label{eq:qilow}
    q_i^{(u)}=\lfloor\alpha ^{\ell_u}\rfloor+4.
\end{equation}

\begin{lemma}\label{lem:telescoping_low}
    It is possible to choose $t_{i+1}>t_i$ in such a way that the following condition hold,
    \begin{enumerate}
        \item $|\widetilde{V}_{i+1}|>\max_{u\in\widetilde{V}_i}\{q_i^{(u)}\}$,
        \item  for all $x\in \widetilde{V}_{i+1}$, \begin{equation}\label{eq:o(x)general}
         \log\left (\frac{\left(\sum_{u\in\widetilde{V}_i} N_{x,u}\right)!}{(\prod_{v\in H_i^{(\overline{u})*}}B_i(x,v)!)(\prod_{u\in\widetilde{V}_i\setminus\{\overline{u}\}}\prod_{v\in H_i^{(u)}}B_i(x,v)!)}\right)\geq \sum_{u\in\widetilde{V}_i}N_{x,u}(\log(A_x/N_{x,u})+\log(q_i^{(u)}-3))+\varepsilon(A_x),
    \end{equation} and
    \item for all $x\in \widetilde{V}_{i+1}$,
$$\gamma_x:=\sum_{u\in\widetilde{V}_i}N_{x,u}\log(A_x/N_{x,u})+\varepsilon(A_x)-\frac{4}{\alpha^{\ell_x}}-\widetilde{L}_i\log(\alpha)(|\widetilde{V}_{i+1}|+1)>0,$$
    \end{enumerate}
     where $H_i^{(\overline{u})*}\setminus \{v_{\min}^i, v_{\max}^i\}$, $\varepsilon(A_x)=O_{|\widetilde{V}_i|,M(i)}(\log(A_x))$ and $M(i)=\max_{u\in\widetilde{V}_i}\{q_i^{(u)}\}$.
\end{lemma}

\begin{proof}
    We prove that by choosing $f$ in \Cref{lem:m_i} growing fast enough and for $t>t_i$ sufficiently large, the conclusion holds when setting $t_{i+1}=t$. We then set $t_{i+1}=t$ for any such an appropriate $t$.\\
    The inequality $|\widetilde{V}_{i+1}|>\max_{u\in\widetilde{V}_i}\{q_i^{(u)}\}$ follows from the fact that the number of vertices of each level in $(V,E)$ is not uniformly bounded (\Cref{prop:unbounded}) and the quantities $q_i^{(u)}$ only depend on the parameters of level $i$, so it suffices to take $t_{i+1}>t_i$ big enough and condition (1) will hold.\\
    To prove condition (2) we apply \Cref{lem:stirling2} to $d=|\widetilde{V}_i|$, 
    $$m_u=\begin{cases}
        q_i^{(u)} & \text{ if } u\neq \overline{u},\\
        q_i^{(u)}-2 & \text{ if } u=\overline{u},
    \end{cases}$$
   $a_{vu}=B_i(x,v)$ for $v\in H_i^{(u)}$, $N_u=N_{x,u}$ and $A=A_x$. Note that this is possible since by construction (telescoping if needed)
   $$\frac{N_{x,u}}{B_i(x,v)}\geq \frac{N_{x,u}}{k_{x,u}+1}\geq m_u-1,$$
   and by choosing a function $f$ in \Cref{lem:m_i} which grows fast enough, the quantities $N_{x,u}$'s are as bigger than $d,m_u$ as we want, which is not affected by telescoping. Applying \Cref{lem:stirling2}, we obtain

\begin{align*}
    \log\left (\frac{\left(\sum_{u\in\widetilde{V}_i} N_{x,u}\right)!}{(\prod_{v\in H_i^{(\overline{u})*}}B_i(x,v)!)(\prod_{u\in\widetilde{V}_i\setminus\{\overline{u}\}}\prod_{v\in H_i^{(u)}}B_i(x,v)!)}\right)\geq & N_{x,\overline{u}}(\log(A_x/N_{x,\overline{u}})+\log(q_i^{(\overline{u})}-3))+\\
    & \sum_{u\in\widetilde{V}_i\setminus\{\overline{u}\}}N_{x,u}(\log(A_x/N_{x,u})+\log(q_i^{(u)}-1))+\varepsilon(A_x)\\
    &\geq \sum_{u\in\widetilde{V}_i}N_{x,u}(\log(A_x/N_{x,u})+\log(q_i^{(u)}-3))+\varepsilon(A_x).
\end{align*} 
 It remains to show that  $\gamma_x>0$ for all $x\in\widetilde{V}_{i+1}$ if $t_{i+1}$ is big enough. Observe that the term $\frac{4}{\alpha^{\ell_x}}$ can be made as small as we want by choosing $t_{i+1}>t_i$ sufficiently big. On the other hand, $\varepsilon(A_x)=O_{|\widetilde{V}_i|,M(i)}(\log(A_x))$ means that there exists a constant $k_i>0$, which depends only on the parameters of the level $i$, such that
    $$-k_i\log(A_x)\leq \varepsilon(A_x)\leq k_i\log(A_x).$$
    For any $x\in \widetilde{V}_{i+1}$, let 
    $$N_{x,\max}=\max\{N_{x,u}:u\in \widetilde{V}_i\},$$
    and let $V_{\max}=\{u\in\widetilde{V}_i:N_{x,u}=N_{x,\max}\}$.\\
    If $\widetilde{V}_i\setminus V_{\max}=\emptyset$, then $N_{x,u}=N_{x,\max}$ for each $u\in\widetilde{V}_i$, and $A_x=|\widetilde{V}_i|N_{x,\max}$. This implies that
    $$\sum_{u\in\widetilde{V}_i}N_{x,u}\log(A_x/N_{x,u})+\varepsilon(A_x)-\frac{4}{\alpha^{\ell_x}}-\widetilde{L}_i\log(\alpha)(|\widetilde{V}_{i+1}|+1)\geq \log(|\widetilde{V}_i|)A_x-k_i\log(A_x)-\frac{4}{\alpha^{\ell_x}}-\widetilde{L}_i\log(\alpha)(|\widetilde{V}_{i+1}|+1).$$
    By choosing $f$ in \Cref{lem:m_i} growing fast enough, for instance $f(n)\geq e^n$, this quantity can be made as big as needed by choosing $t_{i+1}>t_i$ sufficiently large, concluding that $\gamma_x>0$.\\
    If $\widetilde{V}_i\setminus V_{\max}\neq\emptyset$, let $N_{x,\sub}=\max\{N_{x,u}:u\in \widetilde{V}_i\setminus V_{\max}\}$. Note that $N_{x,\sub}\leq A_x/2$. We distinguish two cases.\\
    {\bf Case 1: $N_{x,\sub}\leq \sqrt{A_x}$.} In this case, $N_{x,u}\leq \sqrt{A_x}$ for all $u\not\in V_{\max}$, which implies that 
 $$
        \sum_{u\in\widetilde{V}_i}N_{x,u}\log(A_x/N_{x,u})+\varepsilon(A_x)-\frac{4}{\alpha^{\ell_x}}-\widetilde{L}_i\log(\alpha)(|\widetilde{V}_{i+1}|+1)\geq $$
        $$
        \sum_{u\in\widetilde{V}_i\setminus V_{\max}}N_{x,u}\log(\sqrt{A_x})+\varepsilon(A_x)-\frac{4}{\alpha^{\ell_x}}-\widetilde{L}_i\log(\alpha)(|\widetilde{V}_{i+1}|+1)
   $$
    and since $\varepsilon(A_x)\geq -k_i\log(A_x)$, we get
    $$\sum_{u\in\widetilde{V}_i}N_{x,u}\log(A_x/N_{x,u})+\varepsilon(A_x)-\frac{4}{\alpha^{\ell_x}}-\widetilde{L}_i\log(\alpha)(|\widetilde{V}_{i+1}|+1)
    \geq$$ 
    $$\log(A_x)\left(\frac{1}{2}\sum_{u\in\widetilde{V}_i\setminus V_{\max}}N_{x,u}-k_i-\frac{4}{\alpha^{\ell_x}}-\widetilde{L}_i\log(\alpha)(|\widetilde{V}_{i+1}|+1)\right).$$
    Again, by choosing $f$ in \Cref{lem:m_i} growing fast enough, 
    $$N_{x,u}-k_i-\frac{4}{\alpha^{\ell_x}}-\widetilde{L}_i\log(\alpha)(|\widetilde{V}_{i+1}|+1)>0 \quad \text{ for all } u\in \widetilde{V}_i,$$
    for $t_{i+1}>t_i$ sufficiently large. This leads to the conclusion that $\gamma_x>0$.\\
    {\bf Case 2: $N_{x,\sub}>\sqrt{A_x}$.} In this case, $N_{x,u}>\sqrt{A_x}$ for all $u\in \widetilde{V}_i\setminus V_{\max}$, which implies that $\sum_{u\in\widetilde{V}_i\setminus V_{\max}}N_{x,u}>\sqrt{A_x}$. Using this together with the fact that $A_x/N_{x,\sub}\geq 2$, we get that
     \begin{multline*}
\sum_{u\in\widetilde{V}_i}N_{x,u}\log(A_x/N_{x,u})+\varepsilon(A_x)-\frac{4}{\alpha^{\ell_x}}-\widetilde{L}_i\log(\alpha)(|\widetilde{V}_{i+1}|+1)
   \\ \geq \log(2)\sqrt{A_x}-k_i\log(A_x)-\frac{4}{\alpha^{\ell_x}}-\widetilde{L}_i\log(\alpha)(|\widetilde{V}_{i+1}|+1).         
     \end{multline*}
    Once again, this quantity can be made as large as needed by choosing $f$ growing fast enough in \Cref{lem:m_i} and $t_{i+1}>t_i$ sufficiently big.\\
    We conclude that an appropriate choice of $f$ in \Cref{lem:m_i} allows us to ensure conditions (1)-(3) by telescoping enough. Observe that the requirement that $f$ grows fast is compatible with condition \eqref{eq:condition_f}, it suffices to take any bijective increasing $f$ growing fast enough.
\end{proof}

Choose $t_{i+1}>t_i$ according to \Cref{lem:telescoping_low}. We have thus defined the sequences $(t_i)_{i\in\N}$ and $(q_i^{(u)})_{i\geq 1, u\in\widetilde{V}_i}$. Note that with this definition, if the function $f$ in \Cref{lem:m_i} verifies $f(n)\geq 2n-1$ for all $n$, then condition \eqref{eq:conditionq_i} of the splitting procedure is satisfied. Indeed, for any $u\in\widetilde{V}_i, 1\leq j\leq |\widetilde{V}_{i+1}|$ and $x\in\widetilde{V}_{i+1}$ such that $x=x_j$,
$$\widetilde{M}_i(x,u)\geq 2|\widetilde{V}_{i+1}|-1\geq |\widetilde{V}_{i+1}|+j-1
\geq q_i^{(u)}+j-1.$$

Note also that the condition $f(n)\geq 2n-1$ is compatible with condition \eqref{eq:condition_f} and with the required growth on $f$ to ensure conditions (2) and (3) of \Cref{lem:telescoping_low}, it suffices to take any bijective increasing $f$ growing fast enough. This allows us to apply the splitting procedure to the diagram $(V,E)$. Let $(V',E')$ be the splitting of $(V,E)$ with respect to $(t_i)_{i\in\N}$ and $(q_i^{(u)})_{i\geq 1, u\in\widetilde{V}_i}$ as defined before. According to \cref{lem:recognizable2}, to prove that it is possible to put an order on the edges of $(V',E')$ such that the sequence of morphisms read on $(V',E')$ satisfies the hypotheses of \cref{lem:recognizable}, it suffices to prove that for each $i\in\N$ and $x\in \widetilde{V}_{i+1}$ one has $\cO(x)\geq q_{i+1}^{(x)}$. \\

Note that for each $x\in\widetilde{V}_{i+1}$, 
$$\cO(x)=\frac{\left(\sum_{u\in\widetilde{V}_i} N_{x,u}\right)!}{(\prod_{v\in H_i^{(\overline{u})*}}B_i(x,v)!)(\prod_{u\in\widetilde{V}_i\setminus\{\overline{u}\}}\prod_{v\in H_i^{(u)}}B_i(x,v)!)}$$
From Equation \eqref{eq:o(x)general} we get that
$$\log(\cO(x))\geq \sum_{u\in\widetilde{V}_i}N_{x,u}(\log(A_x/N_{x,u})+\log(q_i^{(u)}-3))+\varepsilon(A_x).$$
Since $q_i^{(u)}=\lfloor \alpha^{\ell_u}\rfloor+4$, $q_i^{(u)}-3=\lfloor \alpha^{\ell_u}\rfloor+1\geq \alpha^{\ell_u}$. Thus,
$$\log(\cO(x))\geq \sum_{u\in\widetilde{V}_i}N_{x,u}(\log(A_x/N_{x,u})+\ell_u\log(\alpha))+\varepsilon(A_x).$$
We want to compare this quantity with $\log(q_{i+1}^{(x)})\leq \ell_x
\log(\alpha)+\frac{4}{\alpha^{\ell_x}}$. Note that
$$\ell_x=(j+1)\ell_{\overline{u}}+\sum_{u\in\widetilde{V}_i}\ell_uN_{x,u},$$
which implies that
$$\log(q_{i+1}^{(x)})\leq \log(\alpha)\left((j+1)L_i+\sum_{u\in\widetilde{V}_i}\ell_uN_{x,u}\right)+\frac{4}{\alpha^{\ell_x}}.$$
Substracting $\log(q_{i+1}^{(x)})$ from $\log(\cO(x))$ we obtain
$$\log(\cO(x))-\log(q_{i+1}^{(x)})\geq \sum_{u\in\widetilde{V}_i}N_{x,u}\log(A_x/N_{x,u})+\varepsilon(A_x)-L_i\log(\alpha)(j+1)-\frac{4}{\alpha^{\ell_x}}\geq \gamma_x>0,$$
where the last inequality comes from condition (3) in \Cref{lem:telescoping_low}. This proves that $\cO(x)\geq q_{i+1}^{(x)}$.\\

Let $\geq$ be an order on $(V',E')$ such that the sequence of read morphisms in $\cB=(V',E',\geq)$ satisfies the hypotheses of \cref{lem:recognizable} and let $(Y,S)$ be the $S$-adic shift associated to the directive sequence of read morphisms on $\cB$. Thanks to \cref{cor:suffconditions_isom}, $(Y,S)$ and $(X_{\cB}, V_{\cB})$ are conjugate, and thus the dimension group $K^0(Y,S)$ is order isomorphic to $K_0(V',E')$, which is in turn order isomorphic to $K^0(X,T)$. On the other hand, the diagram $(V',E')$ is pseudo-ERS (see \cref{rem:preservesers2}), therefore $(Y,S)$ is pseudo-Toeplitz.\\
Recall from \cref{rem:preservesers2} that $L_i'=\widetilde{L}_i=L_{t_i}$ for all $i$. By construction, each morphism read on the diagram $(V',E')$ is injective on words, so we have that for each $i\geq 1$,
$$p_Y(\widetilde{L}_i)\geq q_i^{(u^*)},$$
where $u^*\in \widetilde{V}_i$ is a vertex having the maximum number of paths from the top vertex. By definition, 
$$q_i^{(u^*)}\geq \alpha^{\ell_u^*}=\alpha^{\widetilde{L}_i}$$
Thus,
$$\frac{\log(p_Y(\widetilde{L}_i))}{\widetilde{L}_i}\geq  \frac{\widetilde{L}_i\log(\alpha)}{\widetilde{L}_i}=\log(\alpha).$$
Taking the limit as $i\to\infty$, we obtain that
$h(Y,S)\geq \log(\alpha)$.\\
Once we have shown that $\liminf_{n\to\infty} \frac{p_Y(n)}{g_n}=0$, it will be clear that $h(Y,S)\leq \log(\alpha)$. Indeed, if $(n_j)_{j\in\N}$ is a subsequence such that $\lim_{j\to\infty}\frac{p_Y(n_j)}{g_{n_j}}=0$, for each $\varepsilon >0$ there exists $J_\varepsilon$ such that for all $j\geq J_\varepsilon$, $\frac{p_Y(n_j)}{g_{n_j}}<\varepsilon$. Fix an $\varepsilon>0$. For $j\in\N$ large enough we will have
$$\log(p_Y(n_j))<\log(\varepsilon)+\log(g_{n_j}).$$
Dividing by $n_j$ and taking the limit as $j$ goes to $\infty$, we get
$$h(Y,S)\leq \log(\alpha).$$
It remains then to prove that $\liminf_{n\to\infty}\frac{p_Y(n)}{g_n}=0$.\\
From \cref{lem:m_i}, setting $\widetilde{m}_i=m_{t_i}=F(L_{t_i})=F(\widetilde{L}_i)$, we obtain that
$$\widetilde{m}_iL_i'\leq (\widetilde{m}_i+1)\ell_i'.$$
This condition implies that any word of length $\widetilde{m}_iL_i'$ in $L_{\bt}$ is contained in the image under $\tau_{[0,i)}$ of some word of length $(\widetilde{m}_i+2)$ in $L_{\bt}^{(i)}$. Thus,
$$p_Y(\widetilde{m}_iL_i')\leq \cP_i(\widetilde{m}_i+2)L_i',$$
where $\cP_i(\widetilde{m}_i+2)$ is the number of words of length $\widetilde{m}_i+2$ in the language $L_{\bt}^{(i)}$. Since $\cP_i(\widetilde{m}_i+2)\leq |V_i'|^{\widetilde{m}_i+2}$ and $|V_i'|\leq |\widetilde{V}_i|q_i^{(u^*)}$ , we obtain
$$\log(p_Y(\widetilde{m}_iL_i'))\leq (\widetilde{m}_i+2)(\log(|\widetilde{V}_i|)+\log(q_i^{(u^*)}))+\log(L_i').$$
Since $q_i^{(u^*)}\leq L_i'\log(\alpha)+\frac{4}{\alpha^{L_i'}}$, substracting  $$\log(g_{\widetilde{m}_iL_i'})=\widetilde{m}_iL_i'(\log(\alpha)+C_{\widetilde{m}_iL_i'})$$ we get
\begin{align*}
   \log\left(\frac{p_Y(\widetilde{m}_iL_i')}{g_{\widetilde{m}_iL_i'}}\right) & \leq (\widetilde{m}_i+2)(\log(|\widetilde{V}_i|)+L_i'\log(\alpha)+\frac{4}{\alpha^{L_i'}})+\log(L_i')-\widetilde{m}_iL_i'(\log(\alpha)+C_{\widetilde{m}_iL_i'}) \\
   & \leq (\widetilde{m}_i+2)(\log(|\widetilde{V}_i|+\frac{4}{\alpha^{L_i'}})+L_i'(2\log(\alpha)+1-\widetilde{m}_iC_{\widetilde{m}_iL_i'}).
\end{align*}

Recall that $\widetilde{m}_i=m_{t_i}=F(L_{t_i})$, $\widetilde{V}_i=V_{t_i}$ and $L_i'=\widetilde{L}_i=L_{t_i}$. Since $L_{t_i}\geq f(|V_{t_i}|)$ and $\frac{4}{\alpha^{L_{t_i}}}$ goes to $0$ as $i\to\infty$, for all $i \in\N$ large enough we have,
$$(F(L_{t_i})+2)(\log(|V_{t_i}|)+\frac{4}{\alpha^{L_{t_i}}})\leq (F(L_{t_i})+2)|V_{t_i}|\leq (F(L_{t_i})+2)f^{-1}(L_{t_i}).$$
From \cref{eq:condition_f} we get that 
$$(F(L_{t_i})+2)\log(|V_{t_i}|)\leq L_{t_i}=L_i'.$$
This yields

$$\log\left(\frac{p_Y(\widetilde{m}_iL_i')}{g_{\widetilde{m}_iL_i'}}\right) \leq L_i'(2\log(\alpha)+2-\widetilde{m}_iC_{\widetilde{m}_iL_i'}).$$
From Condition \eqref{eq:suffcondlow}, defining $\lambda=2\log(\alpha)+2-K<0$ we conclude that 
$$\log\left(\frac{p_Y(\widetilde{m}_iL_i')}{g_{\widetilde{m}_iL_i'}}\right) \leq \lambda L_i',$$
which diverges to $-\infty$ as $i$ goes to $\infty$. This proves that
$$\lim_{i\to\infty}\frac{p_Y(\widetilde{m}_iL_i')}{g_{\widetilde{m}_iL_i'}}=0,$$
and thus $\liminf_{n\to\infty}\frac{p_Y(n)}{g_n}=0$. This concludes the proof of the theorem.

\end{proof}

\begin{proof}[Proof of \cref{theo:main}] Let $(X,T)$ be a minimal Cantor system and let $(V,E)$ be a simple Bratteli diagram such that $K_0(V,E)\cong K^0(X,T)$. Thanks to \cref{theo:pseudoers}, we may assume that $(V,E)$ is pseudo-ERS and also that in $(V,E)$, the number of paths from the top vertex to any vertex at level $V_i$ is greater than $f(|V_i|)$ for any increasing function $f$. 

Choose and fix a strictly decreasing sequence $(\alpha_i)_{i\in\N}$ such that $\lim_{i\to\infty}\alpha_i=\alpha$. \\

For a given sequence of telescoping depths $(t_i)_{i\in\N}$, $(\widetilde{V},\widetilde{E})$ the telescoping of $(V,E)$ to $(t_i)_{i\in\N}$ and $i\geq 1$, choose and fix a vertex $\overline{u}\in\widetilde{V}_i$ as in \Cref{sec:splitting}. For $u\in\widetilde{V}_i$ and $x=x_j\in \widetilde{V}_{i+1}$ with $1\leq j\leq |\widetilde{V}_{i+1}|$, let $N_{x,u}$ and $A_x$ be defined as in the proof of \Cref{theo:low}.\\

Once again, we inductively define the sequences $(t_i)_{i\in\N}$ and $(q_i^{(u)})_{i\geq 1, u\in\widetilde{V}_i}$ and then apply to $(V',E')$ the splitting procedure described in \cref{sec:splitting}.\\

Let $t_0=0, t_1=1$. Let $i\geq 1$ and suppose we have defined $t_{i}$, and thus $\widetilde{V}_i$ is defined. For each $u\in\widetilde{V}_i$, define
\begin{equation}
    q_i^{(u)}=\lfloor \alpha_i^{\ell_u}\rfloor+4.
\end{equation}

The following lemma plays the role of \Cref{lem:telescoping_low} in the proof of \Cref{theo:low}. Its proof is very similar and uses the same kind of arguments. The only substantial difference is condition (4) below.

\begin{lemma}\label{lem:telescoping}
It is possible to choose $t_{i+1}>t_{i}$ in such a way that, for all $x\in \widetilde{V}_{i+1}$, the following conditions hold,
\begin{enumerate}
    \item $|\widetilde{V}_{i+1}|>\max_{u \in \widetilde{V}_i}\{ q_i^{(u)}\}$,
    \item  for all $x\in \widetilde{V}_{i+1}$, \begin{equation}\label{eq:o(x)general_1}
         \log\left (\frac{\left(\sum_{u\in\widetilde{V}_i} N_{x,u}\right)!}{(\prod_{v\in H_i^{(\overline{u})*}}B_i(x,v)!)(\prod_{u\in\widetilde{V}_i\setminus\{\overline{u}\}}\prod_{v\in H_i^{(u)}}B_i(x,v)!)}\right)\geq \sum_{u\in\widetilde{V}_i}N_{x,u}(\log(A_x/N_{x,u})+\log(q_i^{(u)}-3))+\varepsilon(A_x),
    \end{equation}
    \item for all $x\in \widetilde{V}_{i+1}$,
$$\delta_x:=\sum_{u\in\widetilde{V}_i}N_{x,u}\log(A_x/N_{x,u})+\varepsilon(A_x)-\frac{4}{\alpha^{\ell_x}}-\widetilde{L}_i\log(\alpha_0)(|\widetilde{V}_{i+1}|+1)>0, \quad \text{ and }$$
    \item $\frac{g(\widetilde{L}_{i+1})}{\alpha_{i+1}^{\widetilde{L}_{i+1}}}<\frac{1}{i+1}$.
\end{enumerate}
\end{lemma}

\begin{proof}
    Conditions (1) and (2) are proved exactly in the same way as conditions (1) and (2) of \Cref{lem:telescoping_low}. Indeed, condition (2) itself is identical to condition (2) of \Cref{lem:telescoping_low} and its proof relies on the appropriate choice of the function $f$ in \Cref{theo:pseudoers}. To prove condition (3), note that the only difference between $\delta_x$ and the quantity $\gamma_x$ defined in \Cref{lem:telescoping_low} is that the term $\widetilde{L}_i\log(\alpha)(|\widetilde{V}_{i+1}|)$ has been replaced by $\widetilde{L}_i\log(\alpha_0)(|\widetilde{V}_{i+1}|)$. Since both $\alpha$ and $\alpha_0$ are fixed constants, the same argument used to prove that $\gamma_x>0$ works to prove that $\delta_x>0$. It remains to prove condition (4). We claim that there exists $N\in\N$ such that $\forall n\geq N$, 
$$\frac{g(n)}{\alpha_{i+1}^n}<\frac{1}{i+1}.$$
If this is not true, then there exists an increasing subsequence $(n_j)_{j\in\N}$ such that for all $j\in \N$, 
$$\frac{g(n_j)}{\alpha_{i+1}^{n_j}}\geq \frac{1}{i+1},$$
which implies that
$$\frac{\log(g(n_j))}{n_j}\geq \log(\alpha_{i+1})+\frac{1}{n_j}\log(1/i+1).$$
Taking the limit $j\to\infty$, we get that $\log(\alpha)\geq \log(\alpha_{i+1})>\log(\alpha)$, a contradiction. Thus, it suffices to take $t_{i+1}>t_i$ large enough so that $L_{t_{i+1}}>N$ and condition (4) will hold.
\end{proof}

Choose $t_{i+1}>t_i$ according to \Cref{lem:telescoping}. We have thus defined the sequences $(t_i)_{i\in\N}$ and $(q_i^{(u)})_{i\geq 1, u\in\widetilde{V}_i}$. Note that, as in the proof of \Cref{theo:low}, with this choice of $(t_i)_{i\in\N}$ and $(q_i^{(u)})_{i\geq 1, u\in\widetilde{V}_i}$, if the function $f$ in \Cref{lem:m_i} verifies $f(n)\geq 2n+1$ for all $n$, then condition \eqref{eq:conditionq_i} of the splitting procedure is satisfied. This requirement is compatible with the growth speed of $f$ required for conditions (2) and (3) of \Cref{lem:telescoping}, and allows us to apply the splitting procedure to the diagram $(V,E)$. Let $(V',E')$ be the splitting of $(V,E)$ with respect to $(t_i)_{i\in\N}$ and $(q_i^{(u)})_{i\geq 1, u\in\widetilde{V}_i}$ as defined before. Following \cref{lem:recognizable2}, to prove that it is possible to put an order on the edges of $(V',E')$ such that the sequence of morphisms read on $(V',E')$ satisfies the hypotheses of \cref{lem:recognizable}, we need to prove that for each $i\in\N$ and $x\in \widetilde{V}_{i+1}$ one has $\cO(x)\geq q_{i+1}^{(x)}$. From Equation \eqref{eq:o(x)general_1} we get that 

$$\log(\cO(x))\geq \sum_{u\in\widetilde{V}_i}N_{x,u}(\log(A_x/N_{x,u})+\log(q_i^{(u)}-3))+\varepsilon(A_x).$$
Since $q_i^{(u)}=\lfloor \alpha_i^{\ell_u}\rfloor+4$, $q_i^{(u)}-3=\lfloor \alpha_i^{\ell_u}\rfloor+1\geq \alpha_i^{\ell_u}$. Thus,
$$\log(\cO(x))\geq \sum_{u\in\widetilde{V}_i}N_{x,u}(\log(A_x/N_{x,u})+\ell_u\log(\alpha_i))+\varepsilon(A_x).$$
We want to compare this quantity with $\log(q_{i+1}^{(x)})\leq \ell_x
\log(\alpha_{i+1})+\frac{4}{\alpha_{i+1}^{\ell_x}}\leq \ell_x\log(\alpha_i)+\frac{4}{\alpha^{\ell_x}}$. Substracting $\log(q_{i+1}^{(x)})$ from $\log(\cO(x))$ we obtain
$$\log(\cO(x))-\log(q_{i+1}^{(x)})\geq \sum_{u\in\widetilde{V}_i}N_{x,u}\log(A_x/N_{x,u})+\varepsilon(A_x)-L_i\log(\alpha_i)(j+1)-\frac{4}{\alpha^{\ell_x}}.$$
Since $\alpha_i<\alpha_0$ and $j\leq |\widetilde{V}_{i+1}|$, we obtain that
$$\log(\cO(x))-\log(q_{i+1}^{(x)})\geq \sum_{u\in\widetilde{V}_i}N_{x,u}\log(A_x/N_{x,u})+\varepsilon(A_x)-L_i\log(\alpha_0)(|\widetilde{V}_{i+1}|+1)-\frac{4}{\alpha^{\ell_x}}=\delta_x>0,$$
where the last inequality comes from condition (3) in \Cref{lem:telescoping}. This proves that $\cO(x)\geq q_{i+1}^{(x)}$.\\

Let $\geq$ be an order on $(V',E')$ such that the sequence of read morphisms in $\cB=(V',E',\geq)$ satisfies the hypotheses of \cref{lem:recognizable} and let $(Y,S)$ be the $S$-adic shift associated to the directive sequence of read morphisms on $\cB$. Thanks to \cref{cor:suffconditions_isom}, $(Y,S)$ and $(X_{\cB}, V_{\cB})$ are conjugate, and thus the dimension group $K^0(Y,S)$ is order isomorphic to $K_0(V',E')$, which is in turn order isomorphic to $K^0(X,T)$. On the other hand, the diagram $(V',E')$ is pseudo-ERS (see \cref{rem:preservesers2}), therefore $(Y,S)$ is pseudo-Toeplitz. In the following we give the bounds on the complexity and topological entropy of $(Y,S)$.\\

Recall that $L_i'=\widetilde{L}_i=L_{t_i}$ for all $i$. Since each morphism is injective on letters, we have that for each $i\geq 1$,
$$p_Y(\widetilde{L}_i)\geq q_i^{(u^*)},$$
where $u^*\in \widetilde{V}_i$ is a vertex having the maximum number of paths from the top vertex. Since $q_i^{(u^*)}=\lfloor \alpha_i^{\ell_{u^*}}\rfloor +4\geq \alpha_i^{\widetilde{L}_i}$, we get
$$\frac{g(\widetilde{L}_i)}{p_Y(\widetilde{L}_i)}<\frac{g(\widetilde{L}_i)}{\alpha_i^{\widetilde{L}_i}}=\frac{g(L_{t_i})}{\alpha_i^{L_{t_i}}}<\frac{1}{i},$$
where the last inequality holds by the last condition in \cref{lem:telescoping}. This implies that 
$$\liminf_{i\to\infty}\frac{g(i)}{p_Y(i)}=0.$$
Note that this automatically implies that $h(Y,S)\geq \log(\alpha)$. Indeed, since $\lim_{i\to\infty}\frac{g(L_i')}{p_Y(L_i')}=0$, for all $\varepsilon>0$ there exists $I_{\varepsilon}\in\N$ such that for all $i\geq I_{\varepsilon}$, $\frac{g(L_i')}{p_Y(L_i')}<\varepsilon$. Fix an $\varepsilon>0$. For $i\in\N$ large enough, we will have
$$\log(g(L_i'))< \log(\varepsilon)+\log(p_Y(L_i')).$$
Dividing by $L_i'$ and taking the limit as $i\to \infty$, we obtain that
$h(Y,S)\geq \log(\alpha)$.\\

Finally, by \cref{theo:sadicentropy}, we know that
$$h(Y,S)\leq \liminf_{i\to\infty}\frac{\log(|V_i'|)}{\ell_i'}.$$
By definition,
$$|V_i'|=\sum_{u\in\widetilde{V}_i} q_i^{(u)}\leq |\widetilde{V}_i|(\alpha_i^{\widetilde{L}_i}+4)=|\widetilde{V}_i|(\alpha_i^{L_i'}+4)$$ 
Thus, 
$$\frac{\log(|V_i'|)}{\ell_i'}\leq\frac{\log(|\widetilde{V}_i|)}{\ell_i'}+\frac{L_i'\log(\alpha_i)}{\ell_i'}+\frac{4}{\ell_i'\alpha_i^{L_i'}}\leq \frac{\log(|\widetilde{V}_i|)}{\ell_i'}+\frac{L_i'\log(\alpha_i)}{\ell_i'}+\frac{4}{\ell_i'\alpha^{L_i'}}.$$
Since in $(V,E)$ the number of paths from the top vertex to any vertex at $V_i$ is at least $|V_i|$ and $\ell_i'=\widetilde{\ell}_i$, by \Cref{theo:pseudoers} the first term goes to zero as $i$ goes to $\infty$ (see also \cref{rem:preservesedgesvsvertices}). Since $(V',E')$ is pseudo-ERS and $\alpha_i$ goes to $\alpha$ as $i$ goes to $\infty$, the second term goes to $\log(\alpha)$ as $i$ goes to $\infty$. Finally, since $\alpha\geq 1$, the third term goes to zero as $i$ goes to $\infty$.
We obtain that $h(Y,S)\leq \log(\alpha)$ and conclude that $h(Y,S)=\log(\alpha)$. This finishes the proof of the theorem.

\end{proof}

Recall that an abelian group $(G,+)$ is called {\em divisible} if for each $g\in G$ and each integer $k\geq 1$, there exists $h\in G$ such that $kh=g$. The following lemma shows that in an SOE class represented by a divisible dimension group it is always possible to find systems admitting ERS Bratteli Vershik representations. Its proof is inspired by the techniques used in \cite{Gjerde_Johansen_Bratteli_Toeplitz:2000} and \cite{Cecchi_Donoso_SOE_superlinear_complexity:2022}.

\begin{lemma}\label{lem:divisibleers}
    Let $(X,T)$ be a minimal Cantor system and let $K^0(X,T)=(G,G^+,u)$ be the dimension group of $(X,T)$. If $G$ is a divisible group, there exists a minimal Cantor system $(\widetilde{X},\widetilde{T})$ verifying
    \begin{itemize}
        \item $K^0(X,T)\cong K^0(\widetilde{X},\widetilde{T})$, and
        \item $(\widetilde{X},\widetilde{T})$ admits an ERS Bratteli-Vershik representation.
    \end{itemize}
\end{lemma}
\begin{proof}
     Let $(V,E)$ be a simple Bratteli diagram such that $K_0(V,E)\cong K^0(X,T)$. Let $(M_i)_{i\in\N}$ be the sequence of incidence matrices of $(V,E)$. Up to telescoping, we may assume that the entries of $M_i$ are greater than $2$ for all $i\geq 1$. Note that we may also assume that $|V_i|\geq 2$ for all $i\geq 1$ and that $|V_1|=2$ with a single edge connecting $v_0$ with each of the two vertices of $V_1$. Indeed, write the first incidence matrix as $M_0=B_0C_0$, where $C_0=(1,1)^t$ and $B_0$ is any $|V_1|\times 2$ positive integer matrix verifying $B_0(j,1)+B_0(j,2)=M_0(j,1)$ for all $1\leq j\leq |V_1|$; for $i\geq 1$, if $|V_i|\geq 2$, define $C_i=\Id_{|V_i|\times |V_i|}$ and $B_i=M_i$, if $|V_i|=1$, define $C_i=(1,1)^t$ and $B_i$ as any $|V_{i+1}|\times 2$ positive integer matrix verifying $B_i(j,1)+B_i(j,2)=M_i(j,1)$ for all $1\leq j\leq |V_{i+1}|$. Then, define $\overline{M_0}=C_0$ and $\overline{M_i}=C_iB_{i-1}$ for all $i\geq 1$. The Bratteli diagram associated with the sequence $(\overline{M_i})_{i\in\N}$ has the desired property and the same dimension group as $(V,E)$.\\
Since we assumed that $M_0=(1,1)^t$, the first incidence matrix has the ERS property. The group $G$ corresponds to the inductive limit $\varinjlim_i (\Z^{|V_i|},M_i)$ and $G$ being divisible, we also have that $\varinjlim_i (\Z^{|V_i|},M_i)\cong \varinjlim_i (\QQ^{|V_i|},M_i)$ (see for instance \cite[Lemma 2.5]{Cecchi_Donoso_SOE_superlinear_complexity:2022}).
We inductively define a sequence $(\widetilde{M}_i)_{i\in\N}$ which will be the sequence of incidence matrices of a Bratteli diagram corresponding to a Bratteli-Vershik representation of the system $(\widetilde{X},\widetilde{T})$. Let $\widetilde{M}_0=M_0$, $J_0=1$, $J_1=\Id_{|V_1|\times |V_1|}$ and $k_1=1$. For $1\leq j\leq |V_2|$, let $\gamma_j$ be the sum of the $j$th row of $M_1$. Let $J_2'$ be the diagonal $|V_2|\times |V_2|$ matrix defined by
$$J_2'=\diag(\gamma_1^{-1},\gamma_2^{-1},\cdots, \gamma_{|V_2|}^{-1}).$$
Let $s_2\in \N^{\ast}$ be such that $s_2J_2'M_1$ has integer entries (for instance, take $s_2$ to be the least common multiple of the denominators of the entries of $J_2'M_1$), and let  
$k_2=2s_2$. Define $J_2=k_2J_2'$ and $\widetilde{M}_1=J_2M_1$. Note that $J_2$ is invertible over $\QQ$ and that $\widetilde{M}_1$ is a positive integer matrix whose entries are all divisible by $2$. 
Note also that $\widetilde{M}_1$ has the ERS property, with $\widetilde{M}_1(1,1,\cdots, 1)^t=(k_2,k_2,\cdots, k_2)^t$.

Let $i>2$. We will define $\widetilde{M}_{i-1}$. Suppose we have defined $J_2,J_3,\cdots, J_{i-1}$, all of them being invertible over $\QQ$. Let $J_i'$ be the diagonal $|V_i|\times |V_i|$ invertible matrix over $\QQ$ such that $J_i'M_{i-1}J_{i-1}^{-1}(1,1,\cdots, 1)^t=(1,1,\cdots, 1)^t$, which means that the $j$th entry of the diagonal of $J_i'$ corresponds to the inverse of the sum of the row $j$ in the matrix $M_{i-1}J_{i-1}^{-1}$. Let $s_i\in \N^{\ast}$ such that  $s_iJ_i'M_{i-1}J_{i-1}^{-1}$ has integer entries. 
Let $k_i=is_i$. Define $J_i=k_iJ_i'$ and $\widetilde{M}_{i-1}=J_iM_{i-1}J_{i-1}^{-1}$.
Note that for all $i\geq 0$, $\widetilde{M}_i$ is a $|V_{i+1}|\times |V_i|$ positive integer matrix which has the ERS property with $\widetilde{M}_i(1,1,\cdots, 1)^t=(k_{i+1},k_{i+1},\cdots, k_{i+1})^t$. Note also that for all $i\geq 1$, $\widetilde{M}_i$ is a positive integer matrix whose entries are all divisible by $i+1$. 

Let $(\widetilde{V},\widetilde{E})$ be the simple Bratteli diagram associated to the sequence of matrices $(\widetilde{M}_i)_{i\in\N}$ and let $(\widetilde{X},\widetilde{T})$ be a minimal Cantor system such that $K^0(\widetilde{X},\widetilde{T})\cong K_0(\widetilde{V},\widetilde{E})$. By construction, $(\widetilde{X},\widetilde{T})$ admits an ERS Bratteli-Vershik representation. Since all matrices $J_i$ are positive and invertible over $\QQ$, the inductive limits $\varinjlim_i (\QQ^{|V_i|},M_i)$ and $\varinjlim_i (\QQ^{|V_i|},\widetilde{M}_i)$ are order isomorphic, and the latter is at the same time order isomorphic to $\varinjlim_i (\Z^{|V_i|},\widetilde{M}_i)$, since all the entries of each $\widetilde{M}_i$ are divisible by $i+1$ (\cite[Lemmas 2.4 and 2.5]{Cecchi_Donoso_SOE_superlinear_complexity:2022}). This proves that $K^0(X,Y)\cong K^0(\widetilde{X},\widetilde{T})$ and concludes the proof of the lemma. 
\end{proof}

\begin{proof}[Proof of \cref{cor:alphaentropymeasureslow}]
Let $K$ be any Choquet simplex $K$ and choose a countable, dense subgroup $G$ of $\Aff(K)$ which is a $\QQ$-vector space and contains the constant function $1$. By \cite[Theorem 3.22]{Hoynes_toeplitz_K_theory:2016} (originally proved in \cite{Effros_dimension_C*_algebras:1981}), $G$ is divisible and $K$ is affine homeomorphic to the set of traces $S(G,G^+,1)$, where the positive cone $G^+$ is defined as in \cite[Theorem 3.22]{Hoynes_toeplitz_K_theory:2016}. Let $(X,T)$ be a minimal Cantor system such that $K^0(X,T)\cong (G,G^+,1)$. Such a system always exists thanks to \cite[Corollary 6.3]{Herman1992}. By \cite[Theorem 5.5]{Herman1992}, $\cM(X,T)\cong S(G,G^+,1)\cong K$. Since the set of invariant measures is preserved under strong orbit equivalence, we may apply \Cref{lem:divisibleers} to get a system $(\widetilde{X},\widetilde{T})$ which admits an ERS Bratteli-Vershik representation, and then apply \cref{theo:low} to $(\widetilde{X},\widetilde{T})$ to get a Toeplitz subshift $(Y,S)$ with the desired properties.
\end{proof}

\begin{proof}[Proof of \cref{cor:zeroentropysoe}]
It follows directly from \cref{theo:main} by taking $\alpha=1$.
\end{proof}

\begin{proof}[Proof of \cref{cor:alphaentropymeasures}]
For any Choquet simplex $K$, let $G$ be a subgroup of $\Aff(K)$  which is a $\QQ$-vector space and contains the constant function $1$, so that $K\cong S(G,G^+,1)$, where $G^+$ is defined as in \cite[Theorem 3.22]{Hoynes_toeplitz_K_theory:2016}. Again, take a minimal Cantor system $(X,T)$ such that $K^0(X,T)\cong (G,G^+,1)$, for which we have $\cM(X,T)\cong S(G,G^+,1)\cong K$. Apply \Cref{lem:divisibleers} to get a system $(\widetilde{X},\widetilde{T})$ which admits an ERS Bratteli-Vershik representstion, and the apply \cref{theo:main} to $(\widetilde{X},\widetilde{T})$ to get a Toeplitz subshift $(Y,S)$ with the desired properties.
\end{proof}

\begin{proof}[Proof of \cref{cor:zeroentropymeasures}]
It follows directly from \cref{cor:alphaentropymeasures} by taking $\alpha=1$.
\end{proof}

\begin{rem}\label{rem:K0divisible}
Note that, since the splitting procedure does not affect the ERS property (see \cref{rem:preservesers2}), \cref{lem:divisibleers} together with \cref{teo:erstoeplitz} imply that both in \cref{theo:low} and \cref{theo:main}, if $K^0(X,T)=(G,G^+,u)$ and $G$ is divisible, then the resulting subshift $(Y,S)$ is a Toeplitz subshift.    
\end{rem}

\bibliographystyle{abbrv}
\bibliography{refs}

\end{document}